\documentclass[11pt,leqno]{article}
\usepackage[utf8]{inputenc}
\usepackage[margin=1.1in]{geometry} 
\usepackage{amsmath,amsthm,amssymb,hyperref, color, mathrsfs, biblatex, enumerate, graphicx, caption, mathtools, dsfont, quiver}

\newtheoremstyle{introdef}% hnamei
{3pt}% hSpace abovei
{3pt}% hSpace belowi
{\itshape}% hBody fonti
{}% hIndent amounti
{\mdseries}% hTheorem head fonti
{}% hPunctuation after theorem headi
{0pt}% hSpace after theorem headi
{}% hTheorem head spec (can be left empty, meaning ‘normal’)i

 \newcommand{\BH}{\mathbb H}

 \newcommand{\CH}{\mathrm{CH}}

\DeclareMathOperator{\AH}{AH}
\DeclareMathOperator{\GF}{GF}
\DeclareMathOperator{\CC}{CC}

\usepackage{relsize}

\newtheorem{theorem}{Theorem}

\newtheorem{introtheorem}{Theorem}

\newtheorem{intromaincor}[introtheorem]{Corollary}

\newtheorem{lemma}{Lemma}
\newtheorem{defn}[lemma]{Definition}
\newtheorem*{lemma*}{Lemma}
\newtheorem*{theorem*}{Theorem}
\newtheorem*{prop*}{Proposition}
\newtheorem{prop}[lemma]{Proposition}
\newtheorem{cor}[lemma]{Corollary}
\newtheorem{maincor}[theorem]{Corollary}
\newtheorem*{cor*}{Corollary}
\newtheorem*{fact*}{Fact}
\newtheorem{fact}[lemma]{Fact}
\newtheorem*{density*}{Density}

\theoremstyle{introdef}
\newtheorem*{empty*}{}

\numberwithin{lemma}{section}

\theoremstyle{definition}

\newcommand{\R}{\mathbb{R}}

\newcommand{\Z}{\mathbb{Z}}
\newcommand{\C}{\mathbb{C}}

\newcommand{\N}{\mathbb{N}}

\title{Connectivity in the space of framed hyperbolic 3-manifolds}
\author{Matthew Zevenbergen}
\date{}

\begin{document}

\maketitle

\begin{abstract}
    We prove that the space $\mathcal{H}_\infty$ of framed infinite volume hyperbolic $3$-manifolds is connected but not path connected. Two proofs of connectivity of this space, which is equipped with the geometric topology, are given, each utilizing the density theorem  for Kleinian groups. In particular, we construct a hyperbolic $3$-manifold whose set of framings is dense in $\mathcal{H}_\infty$. Examples of paths in $\mathcal{H}_\infty$ are discussed, including paths of geometrically finite manifolds limiting to certain infinite type geometric limits of quasi-Fuchsian manifolds. The discussion of paths culminates in describing an infinite family of non-tame hyperbolic $3$-manifolds, each of whose set of framings is a path component of $\mathcal{H}_\infty$, establishing that $\mathcal{H}_\infty$ is not path connected.
\end{abstract}

\section{Introduction}

We will examine questions of global connectivity in the space $\mathcal{H}$ of isometry classes of complete oriented connected framed hyperbolic $3$-manifolds $(M,f)$. This space is equipped with the geometric topology. Informally, two framed manifolds are close in the geometric topology if they are almost isometric on large compact neighborhoods of their framed basepoints. The space $\mathcal{H}$ is homeomorphic to the space $\mathcal{D}$ of Kleinian groups equipped with the Chabauty topology, which we will also refer to as the geometric topology; here, a Kleinian group is a\textit{ torsion free }discrete subgroup of $\mathrm{PSL}_2\C$ (see Section \ref{subsec-GeometricTopology}). 

\begin{empty*}
    A sequence of Kleinian groups $\{\Gamma_n\}\subseteq\mathcal{D}$ converges to $\Gamma$ in the geometric topology on $\mathcal{D}$ if and only if the following two conditions hold:
    \begin{enumerate}
        \item If $\psi\in\mathrm{PSL}_2\C$ is an accumulation point of a sequence $\psi_n\in\Gamma_n$, then $\psi\in\Gamma$.
        \item For all $\psi\in\Gamma$, there exists a sequence $\psi_n\in\Gamma_n$ so that $\psi_n\rightarrow\psi$ in $\mathrm{PSL}_2\C$.
    \end{enumerate}
\end{empty*}

The resolution of the density conjecture (now theorem) by Namazi and Souto \cite{NAM} and separately Ohshika \cite{OHS} provides a starting place for studying the global topology of the space $\mathcal{D}$ of Kleinian groups. A strong version of the density theorem \cite[Corollary 12.3]{NAM} says that every finitely generated Kleinian group is the geometric limit of a sequence of geometrically finite Kleinian groups. Any Kleinian group is the geometric limit of a sequence of finitely generated subgroups of itself (see Lemma \ref{fglim}), so we have the following corollary of the density theorem.

\begin{density*}[Namazi-Souto \cite{NAM}, Ohshika \cite{OHS}]
    Every Kleinian group is the geometric limit of a sequence of geometrically finite Kleinian groups.
\end{density*}

We will use this result to identify all connected components of $\mathcal{H}$. For a hyperbolic $3$-manifold $M$ with positively oriented orthonormal frame bundle $\mathcal{F}M$, we define \[\ell(M)=\{(M,f)\;|\;f\in\mathcal{F}M\}\] to be the \textit{leaf} of $\mathcal{H}$ corresponding to $M$. 

\begin{introtheorem}
\label{introthm-connectedcomponents}
    The connected components of $\mathcal{H}$ are:
    \begin{enumerate}
        \item $\ell(M)$ for each hyperbolic $3$-manifold $M$ with $\mathrm{vol}(M)<\infty$, and
        \item $\mathcal{H}_\infty=\{(M,f)\in\mathcal{H}\;|\;\mathrm{vol}(M)=\infty\}$.
    \end{enumerate}
\end{introtheorem}

The fact that each $\ell(M)$ is a connected component of $\mathcal{H}$ for $M$ with $\mathrm{vol}(M)<\infty$ follows from Mostow-Prasad rigidity and classical results on sequences in the geometric topology (see, e.g., \cite[Section E]{BEN}). To determine that the subset $\mathcal{H}_\infty$ of infinite volume framed hyperbolic $3$-manifolds is connected, we first show that all infinite volume geometrically finite framed hyperbolic $3$-manifolds lie in the same path component of $\mathcal{H}_\infty$ as $(\mathbb{H}^3,\mathcal{O})$, where $\mathcal{O}\in\mathcal{F}\mathbb{H}^3$ (see Lemma \ref{pathcomp}). The density theorem implies that this path component is dense in $\mathcal{H}_\infty$, so the closure of this path component must be $\mathcal{H}_\infty$, which is therefore connected.\par

The strategy to show that $\mathcal{H}_\infty$ is connected is to show that it is the closure of a path connected subset of $\mathcal{H}$, analogous to the proof that the topologist's sine curve is connected. Each leaf $\ell(M)$ is path connected, so the following theorem produces another dense path connected subset of $\mathcal{H}_\infty$, providing a second proof that $\mathcal{H}_\infty$ is connected.

\begin{introtheorem}
\label{introthm-denseleaf}
    There exists a hyperbolic $3$-manifold $M$ such that the leaf $\ell(M)$ is dense in $\mathcal{H}_\infty$.
\end{introtheorem}

This theorem is established by modifying a circle packing construction by Fuchs, Purcell, and Stewart \cite{PUR} (see also Brooks \cite{BRO}) to show that every infinite volume framed hyperbolic $3$-manifold is the geometric limit of a sequence of framed hyperbolic $3$-manifolds whose convex core boundary is a disjoint union of totally geodesic thrice punctured spheres (see Proposition \ref{approxwithsph}). These manifolds are then glued together along totally geodesic thrice punctured spheres in their convex core boundaries to construct the manifold $M$ in Theorem \ref{introthm-denseleaf}.\par

Paths in the geometric topology are more controlled than sequences. Understanding this control is one of our major goals. The following proposition (see Section \ref{subsec-PathConstraints}) provides a way to track elements of $\mathrm{PSL}_2\C$ through a path in $\mathcal{D}$, and is a key tool for studying paths in $\mathcal{D}$.

\begin{prop*}
    \label{introprop-JMaps}
    Let $\Gamma:I\rightarrow\mathcal{D}$ be a path, and set $\Gamma(t)=\Gamma_t$ for all $t$ in the interval $I$. Then, for all $s\in I$, there exists a unique family of maps $J_{s,t}:\Gamma_s\rightarrow\overline{\mathrm{PSL}_2\C}=\mathrm{PSL}_2\C\cup\infty$, defined for all $t\in I$, such that for all $\psi\in\Gamma_s$,
    \begin{enumerate}
        \item $J_{s,t}(\psi)\in\Gamma_t\cup\infty$,
        \item $J_{s,s}(\psi)=\psi$,
        \item the map $t\mapsto J_{s,t}(\psi)\in\overline{\mathrm{PSL}_2\C}$ is continuous for $t\in[0,1]$,
        \item if $s<t'<t$ or $t<t'<s$, then $J_{s,t'}(\psi)=\infty$ implies $J_{s,t}(\psi)=\infty$.
    \end{enumerate}
    Moreover, for any $H\leq\Gamma_s$, $J_{s,t}|_H$ is an injective homomorphism into $\mathrm{PSL}_2\C$ if $\infty\not\in J_{s,t}(H)$.
\end{prop*}

The continuity in property (3) of this proposition tells us that any element $\psi\in\Gamma_s$ can be tracked through the groups $\Gamma_t$ by the path $J_{s,t}(\psi)$ for $t$ close to $s$, and the only way to ``introduce a new element" is for such a path to come in from infinity. One can use the Klein-Maskit combination theorems (see eg. \cite{MAT}) to construct examples of paths in $\mathcal{D}$ exhibiting this behavior, where paths $t\mapsto J_{s,t}(\psi)$ come in from (or diverge to) infinity. In conjunction with the classical theory of deformations of Kleinian groups due to Ahlfors, Bers, Marden, and Sullivan (\cite{MAT}), this allows us to construct the following family of examples, which appears below as Example \hyperlink{Ex-NonTame}{B.3}.

\begin{introtheorem}
\label{introthm-NonTameLimit}
    For a closed surface $S$ of genus at least two, let $C\subseteq S\times\Z\subseteq S\times \R$ be a collection of disjoint simple closed curves so that no two curves in $C$ are isotopic in $S\times \R$ and the curves $C_n=C\cap(S\times\{n\})$ give a pants decomposition of the surface $S_n=S\times\{n\}$ for all $n\in \Z$. Then, there exists a path $G:[1,\infty]\rightarrow\mathcal{D}$ so that $G(t)$ is geometrically finite for all $t<\infty$ and $\mathbb{H}^3/G(\infty)$ is isometric to the unique hyperbolic structure on $(S\times \R)\backslash C$.
\end{introtheorem}

The manifolds $\BH^3 / G(\infty)$ appearing in Theorem \ref{introthm-NonTameLimit} are among those non-tame hyperbolic $3$-manifolds that Thurston approximates with sequences of quasi-Fuchsian manifolds in \cite[Section 7]{THUII}, though the techniques and goals here are quite different.\par

We conclude with a discussion of path connectivity in the space $\mathcal{H}_\infty$ of infinite volume framed hyperbolic $3$-manifolds. We define an infinite family of non-tame hyperbolic $3$-manifolds, which we call \textit{symmetric infinite type }$(G,N)$-\textit{glued hyperbolic }$3$\textit{-manifolds}. Imprecisely, manifolds in this family are constructed by gluing together homeomorphic copies of a compact, oriented, connected, irreducible, atoriodal, acylindrical $3$-manifold $N$ according to vertex adjacencies in a highly symmetric infinite graph $G$ (see Section \ref{subsec-pathcomps} for a precise definition). These manifolds are hyperbolizable by a theorem of Souto and Stover \cite{SOUSTO}. In fact, the symmetry of $G$ implies that each symmetric infinite type $(G,N)$-glued hyperbolic $3$-manifold has a unique hyperbolic structure, by a forthcoming result of Cremaschi and Yarmola \cite{CREMYAR} (see Proposition \ref{Prop-GGluedUnique}). Our main result is the following.

\begin{introtheorem}
    \label{introthm-GGluePathComp}
    For any symmetric infinite type $(G,N)$-glued hyperbolic $3$-manifold $M$, the leaf $\ell(M)$ is a path component of $\mathcal{H}_\infty$.
\end{introtheorem}

Roughly, the proof proceeds by supposing that $\Gamma:[0,1]\rightarrow\mathcal{H}$ is a path so that $M=\mathbb{H}^3/\Gamma(0)$ is a symmetric infinite type $(G,N)$-glued hyperbolic $3$-manifold. Then, one considers a sequence of subgroups $H_n\leq H_{n+1}\leq\Gamma(0)$ exhausting $\Gamma(0)$ corresponding to an exhaustion $M=\cup X_n$, where each $X_n$ is a connected union of copies of $N$. If the image of any map $J_{0,t}|_{H_n}$ contains $\infty$, then one obtains a contradiction to Thurston's ``$\AH(\text{acylindrical})$ is compact" theorem \cite{THUI}, using that each $X_n$ is acylindrical. One then concludes that $J_{0,t}$ must be an injective homomorphism for all $t$, so the rigidity of $M$ implies that each $\Gamma(t)$ contains a subgroup that is conjugate to $\Gamma(0)$. Finally, a result on discrete extensions of Kleinian groups (Proposition \ref{Thm-DiscCountable}) implies that each $\Gamma(t)$ itself must in fact be conjugate to $\Gamma(0)$.\par

Theorem \ref{introthm-GGluePathComp} has the following immediate corollary.

\begin{intromaincor}
\label{introcor-notpathconn}
    $\mathcal{H}_\infty$ is not path connected.
\end{intromaincor}

Our analysis here combines elements  of the classical theory of deformations of infinite volume hyperbolic $3$-manifolds with the study of Chabauty spaces of subgroups of Lie groups. Questions similar to ours have been been considered in each of these settings. In particular, Warakkagun \cite{WAR} has examined the two-dimensional version of our setting, proving that the Chabauty space of torsion free discrete subgroups of $\mathrm{PSL}_2\R$ is path connected, in direct contrast to Corollary \ref{introcor-notpathconn}. Additionally, Baik and Clavier \cite{BAIK}, have studied geometric limits of cyclic subgroups of $\mathrm{PSL}_2\R$, later generalizing to geometric limits of abelian subgroup of $\mathrm{PSL}_2\C$ \cite{BAIK2}. Biringer, Lazarovich, and Leitner \cite{IANCHAB} have further analyzed the space of closed subgroups of $\mathrm{PSL}_2\R$, focusing on global topology. We also note that a result of Fr\c{a}czyk and Gelander \cite{FRAGEL} implies that the space of discrete subgroups of $\mathrm{SL}_n\R$ for $n\geq 3$ that act with infinite covolume on the symmetric space $\mathrm{SL}_n\R/\mathrm{SO}(n)$ is connected, providing a result analogous to Theorem \ref{introthm-connectedcomponents}.\par\par

The classical study of deformation spaces of infinite volume hyperbolic $3$-manifolds focuses on the space $\AH(\Gamma)$ of conjugacy classes of discrete and faithful representations of a Kleinian group $\Gamma$ into $\mathrm{PSL}_2\C$ (see Section \ref{Sec-Reps}). Sullivan's rigidity theorem \cite{SUL} identifies the components of the interior of $\AH(\Gamma)$. Anderson and Canary \cite{ACBumping} show that the intersection of closures of components of the interior of $\AH(\Gamma)$ may be non-empty in a phenomenon called self-bumping, which they study more thoroughly with McCullough in \cite{ACM}. Bromberg \cite{KEN}, and later Brock, Bromberg, Canary, and Minsky \cite{BBCM}, study local connectivity in $\AH(\Gamma)$.\par

We now describe the organization of the rest of the paper. In Section \ref{sec-background} we give background information on Kleinian groups, hyperbolic $3$-manifolds, the geometric topology, and spaces of representations of Kleinian groups. Section \ref{subsec-comps} focuses on determining the connected components of $\mathcal{H}$, in particular proving Theorem \ref{Thm-connectedcomponents}. Section \ref{subsec-denseleaf} constructs a manifold $M$ so that the leaf $\ell(M)$ is dense in $\mathcal{H}_\infty$, proving Theorem \ref{thm-denseleaf}. In Section \ref{sec-paths}, we turn our attention to paths. Examples of paths, including those in Theorem \ref{introthm-NonTameLimit}, are produced in Section \ref{subsec-pathexamples}. The basic machinery for discussing paths in $\mathcal{D}$ is developed in Section \ref{subsec-PathConstraints}, in particular constructing the maps $J_{s,t}$. Finally, we conclude with Section \ref{subsec-pathcomps}, in which we define $(G,N)$-glued hyperbolic $3$-manifolds and prove Theorem \ref{Thm-GGluePathComp}, establishing Corollary \ref{cor-notpathconn} that $\mathcal{H}_\infty$ is not path connected.\par

\vspace{0.1in}

\textbf{Acknowledgments:} The author would like to thank Ian Biringer for numerous conversations and suggestions, and for his unending kindness. The author thanks Tommaso Cremaschi and Andrew Yarmola for allowing him to include a result from a forthcoming paper \cite{CREMYAR}, appearing here as Proposition \ref{Prop-GGluedUnique}. The author also thanks Cremaschi for helpful conversations.

%%%%%%%%%%%%%%%%%%%%%%%%%%%%%%%%%%%%%%%%%%%%%%%%%%%%%%%%%%%%%%%%%%%%%%%%%%%%%%%%%%%%%%%%%%%%%%%%%%%%%%%%%%%%%%%%%%%%%%%%%%%%%%%%%%%%%%%%%%%%%%

\section{Background}
\label{sec-background}

Throughout the paper, all manifolds considered are connected and oriented. Any hyperbolic manifold will be complete unless otherwise noted.

%%%%%%%%%%%%%%%%%%%%%%%%%%%%%%%%%%%%%%%%%%%%%%%%%%%%%%%%%%%%%%%%%%%%%%%%%%%%%%%%%%%%%%%%%%%%%%%%%%%%%%%%%%%%%%%%%%%%%%%%%%%%%%%%%%%%%%%%%%%%%%

\subsection{Kleinian groups and hyperbolic 3-manifolds}

\label{KGroupBackground}
Identifying the Riemann sphere with the ideal boundary $S^2_\infty$ of $\mathbb{H}^3$, Poincar\'{e} extensions of M\"{o}bius transformations allow us to identify $\mathrm{PSL}_2\C$ and $\mathrm{Isom}^+(\mathbb{H}^3)$, the group of orientation preserving isometries of hyperbolic $3$-space. A \textbf{Kleinian group} is a discrete and torsion free subgroup of $\mathrm{PSL}_2\mathbb{C}$; note that in many references a Kleinian group is not required to be torsion free, but we include this in our definition. We let $\mathcal{D}$ denote the set of Kleinian groups: \[\mathcal{D}=\{\Gamma\leq\mathrm{PSL}_2\C \;|\; \Gamma\text{ is discrete and torsion free}\}.\]\par

A \textbf{hyperbolic }$3$\textbf{-manifold} $M$ is the quotient $\mathbb{H}^3/\Gamma$ for some Kleinian group $\Gamma$. From $\Gamma$, $M$ inherits a $(\mathrm{PSL}_2\C,\mathbb{H}^3)$-structure, in the language of $(G,X)$-structures (see \cite{THUR}, \cite{BEN}). For any $p\in M$, we obtain a \textbf{holonomy representation} of $\pi_1(M,p)$ with respect to $\Gamma$, an isomorphism $\Psi:\pi_1(M,p)\rightarrow \Gamma$. The holonomy representation $\Psi$ is unique up to conjugation by $\Gamma$, and is determined by choosing a $\Gamma$-lift of $p$ to $\mathbb{H}^3$. Hyperbolic $3$-manifolds will be identified if there is an orientation preserving isometry between them.\par

For a Kleinian group $\Gamma$, the \textbf{limit set} $\Lambda(\Gamma)$ of $\Gamma$ is constructed by selecting a point $x\in\mathbb{H}^3$ and letting $\Lambda(\Gamma)=\overline{\Gamma \cdot x}\cap S^2_\infty$. This definition is independent of the choice of $x$. Additionally, define the \textbf{domain of discontinuity} of $\Gamma$ as $\Omega(\Gamma)=S^2_\infty\backslash\Lambda(\Gamma)$, and  recall that $\Omega(\Gamma)$ is the largest subset of $S^2_\infty$ on which $\Gamma$ acts properly discontinuously.\par

Let $\CH(\Gamma)\subseteq\mathbb{H}^3$ be the (hyperbolic) convex hull of $\Lambda(\Gamma)$. Then, for the hyperbolic $3$-manifold $M=\mathbb{H}^3/\Gamma$, we define the \textbf{convex core} of $M$ to be $\CC(M)=\CH(\Gamma)/\Gamma\subseteq M$. $M$ and $\Gamma$ are each called \textbf{geometrically finite} if $\Gamma$ is finitely generated and $\CC(M)$ has finite volume; $M$ and $\Gamma$ are called \textbf{convex cocompact} if $\CC(M)$ is compact. Finally, $M$ and $\Gamma$ are said to be \textbf{elementary} if $|\Lambda(\Gamma)|\leq 2$, which is equivalent to $\Gamma$ being abelian. See \cite{NOTE} and \cite{MAT} for more background on Kleinian groups.\par

A subset $U$ of either $S^2_\infty$ or $\mathbb{H}^3$ is said to be \textbf{precisely invariant} for a subgroup $H\leq\Gamma$ if $\psi(U)=U$ for all $\psi\in H$, and $\psi(U)\cap U=\varnothing$ for all $\psi\in\Gamma\backslash H$. For $p\in M=\mathbb{H}^3/\Gamma$, the \textbf{injectivity radius} of $M$ at $p$ is \[\mathrm{inj}_M(p) = \sup\{R>0 : B_{\mathbb{H}^3}(\tilde{p},R)\text{ is precisely invariant for }\{\mathds{1}\}\text{ in }\Gamma\}\] where $\tilde{p}\in\mathbb{H}^3$ is a $\Gamma$-lift of $p$ and $\mathds{1}\in\mathrm{PSL}_2\C$ is the identity. In other words, $\mathrm{inj}_M(p)$ is the supremal radius of hyperbolic balls centered at $p\in M$ that isometrically embed into $M$. \par

A \textbf{frame} $f$ for a hyperbolic $3$-manifold $M$ at $p\in M$ is an ordered positively oriented orthonormal basis for  the tangent space $T_pM$. For a fixed $p\in M$, let $F_pM$ be the space of frames for $M$ at $p$. When we would like to make explicit the basepoint underlying a frame in $F_pM$, we write the frame as $f_p$. Let $\mathcal{F}M$ be the bundle of positively oriented orthonormal frames over $M$. As a set, for any $U\subseteq M$, in particular when $U=M$, \[\mathcal{F}U = \bigcup_{p\in U}F_pM.\] A pair $(M,f)$ where $f\in\mathcal{F}M$ is a \textbf{framed hyperbolic $3$-manifold}. Let \[\mathcal{H}=\{(M,f)\;|\; M \text{ hyperbolic }3\text{-manifold}, f\in \mathcal{F}M\}/\sim\] where $(M,f)\sim (N,h)$ if there is an isometry $\varphi:M\rightarrow N$ such that the induced map $\varphi_*:\mathcal{F}M\rightarrow\mathcal{F}N$ satisfies $\varphi_*(f)=h$. We will refer to equivalence classes $[(M,f)]\in \mathcal{H}$ by their representatives, for example $(M,f)$. Background on $\mathcal{H}$ can be found in \cite{BEN}.\par

Once and for all, fix $O\in\mathbb{H}^3$ and $\mathcal{O}_O\in F_O\mathbb{H}^3$.  We can now define a map $\Phi:\mathcal{D}\rightarrow \mathcal{H}$ by \[\Phi(\Gamma)=(\mathbb{H}^3/\Gamma,\pi_\Gamma(\mathcal{O}_O))\] where $\pi_\Gamma:\mathcal{F}\mathbb{H}^3\rightarrow\mathcal{F}(\mathbb{H}^3/\Gamma)$ is the natural projection. As discussed in \cite[E.1.9]{BEN}, $\Phi$ is in fact a bijection, which is a primary reason for working with framed hyperbolic $3$-manifolds, rather than, say, pointed hyperbolic $3$-manifolds.\par

\subsection{The geometric topology}
\label{subsec-GeometricTopology}

We will describe a topology on $\mathcal{H}$. First, we make the following definition:

\begin{defn}
    \label{eRclosedef2}
   Let $(M,f_p)=(\mathbb{H}^3/\Gamma_1,\pi_{\Gamma_1}(\mathcal{O}_O))$ and $(N,h_q)=(\mathbb{H}^3/\Gamma_2,\pi_{\Gamma_2}(\mathcal{O}_O))$ be framed hyperbolic $3$-manifolds. We say $(N,h_q)$ is $(\varepsilon,R)$\textbf{-close} to $(M,f_p)$ if there is a $(1+\varepsilon)$-bilipshitz embedding $\tilde{g}:B_{\mathbb{H}^3}(O,R)\rightarrow\mathbb{H}^3$ such that 
    \begin{enumerate}
        \item $\tilde{g}(O)=O$,
         \item $D_{C^{0}}(\tilde{g},\mathds{1}|_{B_{\mathbb{H}^3}(O,R)})<\varepsilon$, and
        \item $\tilde{g}$ descends to an embedding $g:B_M(p,R)\rightarrow N$.
    \end{enumerate}
\end{defn}

Here, for $U\subseteq\mathbb{H}^3$ and $g_1,g_2:U\rightarrow\mathbb{H}^3$,

\[D_{C^0}(g_1,g_2)=\sup_{z\in U}d_{\mathbb{H}^3}(g_1(z),g_2(z)).\]

\begin{defn}
\label{geomtopdef}
    The \textbf{geometric topology} on $\mathcal{H}$ is the topology generated by taking the collection of sets of the form 
    \[\{(N,q)\in\mathcal{H}\;|\;(N,q) \text{ is }(\varepsilon,R)\text{-close to }(M,p)\},\]
    where $(M,p)\in\mathcal{H}$ and $\varepsilon,R>0$ as a subbasis.
\end{defn}

Similar, definitions of $(\varepsilon,R)$-close are given in \cite{BEN}, \cite{NOTE}, and \cite{PUR}. This topology is metrizable (see \cite[E.1.4]{BEN} and \cite{IANMET}). The following lemma provides a type of transitivity for being $(\varepsilon,R)$-close, and follows immediately from Definition \ref{eRclosedef2}.\par

\begin{lemma}
    \label{lemma-eRTransitivity}
    Suppose $(N,h_q)$ is $(\varepsilon,R)$-close to $(M,f)$, and that for some hyperbolic $3$-manifold $N'$, there exists an isometric embedding $\iota: B_N(q,R)\rightarrow N'$. Then, $(N',\iota_*(h))$ is $(\varepsilon,R)$-close to $(M,f)$, where $\iota_*:\mathcal{F}B_N(q,R)\rightarrow\mathcal{F}N'$ is the induced map on frame bundles.
\end{lemma}

If $\mathcal{D}$ is equipped with the topology inherited from the Chabauty topology (see \cite[Section E]{BEN}) on closed subsets of $\mathrm{PSL}_2\C$, the map $\Phi$ is a homeomorphism. For clarity, we will therefore also refer to the Chabauty topology on $\mathcal{D}$ as the geometric topology. The following fact (see \cite[Proposition E.1.2]{BEN}) describes sequential convergence in the geometric topology on $\mathcal{D}$.

\begin{fact}
\label{ChabProps}
    A sequence $\{\Gamma_n\}\subseteq\mathcal{D}$ converges to $\Gamma$ in the geometric topology on $\mathcal{D}$ if and only if the following two conditions hold:
    \begin{enumerate}
        \item If $\psi\in\mathrm{PSL}_2\C$ is an accumulation point of a sequence $\psi_n\in\Gamma_n$, then $\psi\in\Gamma$.
        \item For all $\psi\in\Gamma$, there exists a sequence $\psi_n\in\Gamma_n$ so that $\psi_n\rightarrow\psi$ in $\mathrm{PSL}_2\C$.
    \end{enumerate}
    The map $\Phi:\mathcal{D}\rightarrow\mathcal{H}$ is a homeomorphism between the geometric topologies on these sets.
\end{fact}

Note that for a convergent sequence $\Gamma_n\rightarrow\Gamma$, the groups $\Gamma_n$ are not required to be isomorphic to each other or to $\Gamma$. Often, the geometric topology is discussed in the context of sequences of representations of a fixed group, but this is not the case here.\par

It will be useful to have the following lemma recorded.

\begin{lemma}
\label{fglim}
    Any Kleinian group is the geometric limit of a sequence of finitely generated subgroups of itself.
\end{lemma}
\begin{proof}
For a Kleinian $\Gamma$, enumerate $\Gamma=\{\psi_1,\psi_2,...\}$ and let $\Gamma_n=\langle\psi_1,...,\psi_n\rangle$. We will show that the sequence $\{\Gamma_n\}$ converges geometrically to $\Gamma$.\par

We just need to check the two conditions in Fact \ref{ChabProps}. Observe that the first condition holds since $\Gamma$ is a closed subset of $\mathrm{PSL}_2\mathbb{C}$ and $\Gamma_n\subseteq\Gamma$ for all $n$. The second condition is immediate, since any $\psi\in\Gamma$ is contained in $\Gamma_n$ for $n$ sufficiently large.
\end{proof}

\subsection{Discrete and faithful representations}
\label{Sec-Reps}

Let $\Gamma$ be a nonelementary Kleinian group. An injective homomorphism $\rho:\Gamma\rightarrow\mathrm{PSL}_2\C$ whose image is a Kleinian group is called a \textbf{discrete and faithful} representation of $\Gamma$ into $\mathrm{PSL}_2\C$, and we let $D(\Gamma)\subseteq\mathrm{Hom}(\Gamma,\mathrm{PSL}_2\C)$ denote the set such representations. Equip $D(\Gamma)$ with the \textbf{algebraic topology} in which a sequence of representations $\{\rho_n\}\subseteq D(\Gamma)$ algebraically converges to $\rho\in D(\Gamma)$ if \[\lim_n \rho_n(\psi)=\rho(\psi)\] for all $\psi\in\Gamma$, where convergence of the limit is in $\mathrm{PSL}_2\C$. An algebraically convergent sequence $\rho_n\rightarrow\rho$ in $D(\Gamma)$ is said to converge \textbf{strongly} if the images of the representations converge geometrically: that is, $\rho_n(\Gamma)\rightarrow\rho(\Gamma)$ geometrically in $\mathcal{D}$. We let $S(\Gamma)$ denote the set $D(\Gamma)$ now equipped with the topology of strong convergence. \par

Let $\AH(\Gamma)=D(\Gamma)/\text{conj}$ denote the quotient of $D(\Gamma)$ by the conjugation action of $\mathrm{PSL}_2\C$, and equip $\AH(\Gamma)$ with the topology inherited from the algebraic topology on $D(\Gamma)$. This topology on $\AH(\Gamma)$ is also referred to as the algebraic topology. Note that a sequence $\{[\rho_n]\}\subseteq \AH(\Gamma)$ algebraically converges to $[\rho]$ if and only if there exists a sequence $\{\psi_n\}\subseteq\mathrm{PSL}_2\C$ such that $\psi_n\rho_n\psi_n^{-1}$ converges to $\rho$ in $D(\Gamma)$. See \cite{MAT} for more details on the algebraic topology.\par

A representation $\rho\in D(\Gamma)$ is type preserving if $\rho(\psi)$ is parabolic exactly when $\psi\in\Gamma$ is parabolic, and we let $D_t(\Gamma)\subseteq D(\Gamma)$ be the subspace of type preserving representations. If $\Gamma$ is geometrically finite, we let $D_{\GF}(\Gamma)$ denote the component of the interior of $D_t(\Gamma)$ containing the identity mapping. It follows from the Marden \cite{MAR} and Sullivan stability theorems \cite{SUL} that $D_{\GF}(\Gamma)$ consists of all representations of $\Gamma$ whose image is a geometrically finite Kleinian group that is quasiconformally conjugate to $\Gamma$: that is, there exists a quasiconformal homeomorphism $F:S^2_\infty\rightarrow S^2_\infty$ such that for all $\psi\in\Gamma$, $\rho(\psi)=F\circ\psi\circ F^{-1}$. If $\Gamma$ is convex cocompact, the image of every representation in $D_{\GF}(\Gamma)$ is convex cocompact. We let $\GF(\Gamma)=D_{\GF}(\Gamma)/\text{conj}$ denote the projection of $D_{\GF}(\Gamma)$ to $\AH(\Gamma)$. \par

We will record a general fact about representations of nonelementary Kleinian groups, which follows from the fact that $\mathrm{PSL}_2\C$ acts uniquely triply transitively on $S^2_\infty$.

\begin{fact}
    \label{Fact-UniqueConjugator}
    Suppose $\rho_1,\rho_2\in D(\Gamma)$ represent the same conjugacy class in $\AH(\Gamma)$ for some nonelementary Kleinian group $\Gamma$. Then, there is a unique $\varphi\in\mathrm{PSL}_2\C$ such that $\rho_2=\varphi\rho_1\varphi^{-1}$. Additionally, if $\rho_1,\rho_2$ are perturbed in $D(\Gamma)$ while remaining conjugate, then $\varphi$ varies continuously.
\end{fact}

The continuity claim at the end of the fact follows from writing $\varphi$ as a continuous function of the fixed points of some elements of each $\rho_i(\Gamma)$, the fixed points being continuous functions of the isometries themselves. From this fact, we can prove the following lemma.

\begin{lemma}
\label{Lemma-AHNonelemSubgroupConvg}
    For a Kleinian group $\Gamma$, let $[\rho_n]\rightarrow[\rho]$ be a convergent sequence in $\AH(\Gamma)$. Assume that there exists a nonelementary subgroup $H\leq\Gamma$ and a representative $\rho_n\in D(\Gamma)$ of each conjugacy class $[\rho_n]$ so that for some $\rho_H\in D(H)$, we have $\rho_n|_H\rightarrow\rho_H$ in $D(H)$. Then, there exists a representative $\rho$ of $[\rho]$ such that $\rho_n\rightarrow\rho$ in $D(\Gamma)$. In particular, $\rho|_H=\rho_H$.
\end{lemma}

\begin{proof}
    Let $\{\psi_n\}\subseteq\mathrm{PSL}_2\C$ be a sequence so that $\psi_n\rho_n\psi_n^{-1}$ converges in $D(\Gamma)$ to some representation $\rho'$. Since the sequence $\{\rho_n|_H\}$ converges, the sequence $\{\psi_n\}$ must be bounded in $\mathrm{PSL}_2\C$, and therefore the sequence $\{\rho_n\}$ must be bounded in $D(\Gamma)$. A theorem of J{\o}rgensen \cite{JOR} tells us $D(\Gamma)$ is closed in $\mathrm{Hom}(\Gamma,\mathrm{PSL}_2\C)$, so it follows that for any subsequence of $\{\rho_n\}$, there exists a further subsequence that converges to some $\rho\in D(\Gamma)$. Observe that this representation $\rho$ is independent of the initial subsequence chosen: $\rho$ must be conjugate in $\mathrm{PSL}_2\C$ to $\rho'$, and for each initial subsequence it follows that $\rho|_H=\rho_H$, so the observation follows from the uniqueness statement in Fact \ref{Fact-UniqueConjugator}. Since every subsequence of $\{\rho_n\}$ has a subsequence that converges to $\rho$, we have $\rho_n\rightarrow \rho$.
\end{proof}

%%%%%%%%%%%%%%%%%%%%%%%%%%%%%%%%%%%%%%%%%%%%%%%%%%%%%%%%%%%%%%%%%%%%%%%%%%%%%%%%%%%%%%%%%%%%%%%%%%%%%%%%%%%%%%%%%%%%%%%%%%%%%%%%%%%%%%%%%%%%%

\section{Connectivity}
\label{sec-connectivity}

This section will focus on connectivity in $\mathcal{H}$. Section \ref{subsec-comps} will focus on establishing that the subspace $\mathcal{H}_\infty=\{(M,f)\;|\;\mathrm{vol}(M)=\infty\}$ is connected, and in Section \ref{subsec-denseleaf} we will construct a non-tame hyperbolic $3$-manifold $M$ such that the leaf $\ell(M)$ (see Section \ref{subsec-comps}) is dense in $\mathcal{H}_\infty$.

\subsection{Connected Components}
\label{subsec-comps}

The goal of this section will be to determine the connected components of $\mathcal{H}$, in particular showing that the subspace $\mathcal{H}_\infty$ is connected. Briefly, the idea will be to use the density theorem to show that there is a path connected subset that is dense in $\mathcal{H}_\infty$.\par

For a hyperbolic $3$-manifold $M$, we have a map $L_M:\mathcal{F}M\rightarrow\mathcal{H}$ such that $L_M(f)=(M,f)$. The set \[\ell(M)=L_M(\mathcal{F}M)=\{(M,f)\in\mathcal{H}\;|\;f\in\mathcal{F}M\}\] is called the \textbf{leaf} of $\mathcal{H}$ corresponding to $M$. Biringer and Ab\'{e}rt discuss a similar leaf decomposition of the space of pointed Riemannian $d$-manifolds in \cite{IAN}. \par

Following from the equivalence relation defining $\mathcal{H}$, $L_M(f)=L_M(h)$ if and only if there is an orientation preserving isometry of $M$ such that the induced map on $\mathcal{F}M$ takes $f$ to $h$. Thus, there is a bijection $\ell(M)\leftrightarrow \mathcal{F}M/\mathrm{Isom}^+(M)$ descending from $L_M$.\par

\begin{prop}
\label{pathprop}

    $L_M$ is continuous for any hyperbolic $3$-manifold $M$.
    
\end{prop}

\begin{proof}
    It will suffice to show that the map $\Phi^{-1}\circ L_M:\mathcal{F}M\rightarrow\mathcal{D}$ is continuous, where $\Phi:\mathcal{D}\rightarrow\mathcal{H}$ is the homeomorphism from Section \ref{KGroupBackground}. Fix a convergent sequence $f_n\rightarrow f$ in $\mathcal{F}M$ and set $\Gamma=\Phi^{-1}((M,f))$. Let $\{\tilde{f}_n\}\subseteq\mathcal{F}\mathbb{H}^3$ be a $\Gamma$-lift of the sequence $\{f_n\}$ so that $\tilde{f}_n\rightarrow\mathcal{O}_O$ as $n\rightarrow\infty$. Recalling that $\mathrm{PSL}_2\C$ acts simply transitively on $\mathcal{F}\mathbb{H}^3$, set $\psi_n\in\mathrm{PSL}_2\C$ so that $\psi_n(\tilde{f}_n)=\mathcal{O}_O$. Then, we have $\psi_n\rightarrow\mathds{1}$ as $n\rightarrow\infty$, so \[(\Phi^{-1}\circ L_M)(f_n)=\psi_n\Gamma\psi_n^{-1}\rightarrow \Gamma = (\Phi^{-1}\circ L_M)(f)\] as $n\rightarrow\infty$.
\end{proof}

Each frame bundle $\mathcal{F}M$ is path connected, so we have the following immediate corollary. 

\begin{cor}
\label{leaf-pathcomp}
    For any hyperbolic $3$-manifold $M$, $\ell(M)$ is path connected in $\mathcal{H}$.
\end{cor}

\begin{lemma}
\label{pathcomp}
    All framed hyperbolic $3$-manifolds $(\mathbb{H}^3/\Gamma,f)\in\mathcal{H}$ such that $\Omega(\Gamma)\neq\varnothing$ are contained in the same path component of $\mathcal{H}$.
\end{lemma}

\begin{proof}
    Let $M=\mathbb{H}^3/\Gamma$, where $\Omega(\Gamma)\neq\varnothing$. We will show that $(M,f)$ is in the same path component of $\mathcal{H}$ as $(\mathbb{H}^3,\mathcal{O}_O)$. Pick $x\in\Omega(\Gamma)$. Since $\Omega(\Gamma)$ is open and $\Gamma$ acts properly discontinuously on $\Omega(\Gamma)$, we may pick an open circular disk $D\subseteq\Omega(\Gamma)$ containing $x$ such that $D$, and hence the convex hull $\CH(D)\subseteq\mathbb{H}^3$, is precisely invariant for $\{\mathds{1}\}$ in $\Gamma$. \par 

    Let $\gamma:[0,\infty)\rightarrow\mathbb{H}^3$ be a geodesic ray such that $\gamma(0)=O$ and $\gamma(t)$ limits to $x$. Since $\CH(D)$ isometrically embeds into $M$, we see that for any $R>0$, $B_{\mathbb{H}^3}(\gamma(t),R)$ isometrically embeds into $M$ for $t$ sufficiently large. Thus, letting $f^t\in\mathcal{F}M$ be the $\Gamma$-projection of the time $t$ parallel transport of $\mathcal{O}_O$ along $\gamma$, Lemma \ref{lemma-eRTransitivity} implies that $(M,f^t)$ geometrically converges to $(\mathbb{H}^3,\mathcal{O}_O)$ as $t\rightarrow\infty$.
\end{proof}

Lemma \ref{pathcomp} provides a large path component of $\mathcal{H}_\infty$. The condition that $\Omega(\Gamma)\neq\varnothing$ occurs for a number of large families of Kleinian groups. In particular, this condition holds if $\Gamma$ is geometrically finite and has infinite covolume. \par

\begin{lemma}
\label{limofgeomfin}
    Any $(M,f)\in\mathcal{H}_\infty$ is the geometric limit of a sequence of geometrically finite elements of $\mathcal{H}_\infty$.
\end{lemma}
\begin{proof}
    Namazi and Souto \cite[Corollary 12.3]{NAM} proved a strong version of the density theorem that tells us that all finitely generated elements of \[\mathcal{D}_\infty=\{\Gamma\in\mathcal{D}\;|\;\mathrm{vol}(\mathbb{H}^3/\Gamma)=\infty\}=\Phi^{-1}(\mathcal{H}_\infty)\] are geometric limits of geometrically finite elements of $\mathcal{D}_\infty$. Now, the result follows from Lemma \ref{fglim}.
\end{proof}

\begin{theorem}
\label{Thm-connectedcomponents}
    The connected components of $\mathcal{H}$ are:
    \begin{enumerate}
        \item $\ell(M)$ for each hyperbolic $3$-manifold $M$ with $\mathrm{vol}(M)<\infty$, and
        \item $\mathcal{H}_\infty$.
    \end{enumerate}
\end{theorem}

\begin{proof}
    Classical results on sequences in the geometric topology (see eg. \cite[Theorem E.2.4]{BEN}) tell us that if $\{(M_i,f^i)\}\subseteq\mathcal{H}$ is any sequence converging to a finite volume $(M,f)\in\mathcal{H}$, then for $i$ sufficiently large, either $M_i$ is isometric to $M$ (in which case $(M_i,f^i)\in\ell(M)$) or $M_i$ has finite volume and strictly fewer cusps than $M$. Hence, one may proceed by induction on $n$ to show that for any finite volume hyperbolic $3$-manifold $M$ with $n$ cusps, $\ell(M)$ is a connected component of $\mathcal{H}$.\par
    
    To show that $\mathcal{H}_\infty$ is its own connected component, it will now suffice to show that $\mathcal{H}_\infty$ is connected: indeed, we can write $\mathcal{H}$ as the disjoint union \[\mathcal{H}=\mathcal{H}_\infty\cup\left(\bigcup_{\mathrm{vol}(M)<\infty}\ell(M)\right),\] where we already know that the leaves in the union are their own connected components. Note that Proposition \ref{pathcomp} tells us all $(M,f)\in\mathcal{H}_\infty$ with $M$ geometrically finite are contained in a single path component, and Lemma \ref{limofgeomfin} then tells us that this path component is dense in $\mathcal{H}_\infty$. Since the closure of a connected set is connected, we conclude that $\mathcal{H}_\infty$ is connected, as desired.
\end{proof}

\subsection{Dense Leaf}
\label{subsec-denseleaf}

In this subsection, we will construct a hyperbolic $3$-manifold $M$ such that the leaf $\ell(M)$ is a dense subset of $\mathcal{H}_\infty$. In other words, any infinite volume framed hyperbolic $3$-manifold may be approximated by appropriate framings on $M$. By Corollary \ref{leaf-pathcomp}, $\ell(M)$ is consequently another dense path connected subset of $\mathcal{H}_\infty$ (cf. Lemma \ref{pathcomp}), providing a second proof that $\mathcal{H}_\infty$ is connected. The construction of $M$ will feature a modification of a construction by Fuchs, Purcell and Stewart in \cite{PUR}, which utilizes the circle packing machinery developed by Brooks \cite{BRO}.\par

\begin{defn}
\label{def-circlepack}
    Let $\Gamma\in\mathcal{D}$ be convex cocompact with infinite covolume. A \textbf{circle packing} $P$ on $\partial_{\infty}(\mathbb{H}^3/\Gamma)=\Omega(\Gamma)/\Gamma$ is a $\Gamma$-invariant collection of (projective) circles on $\Omega(\Gamma)$ together with a triangulation $V$ of $\Omega(\Gamma)$ with $P=\{c_v\;|\;v\text{ a vertex of }V\}$ satisfying the following:
    \begin{enumerate}
        \item The circles of $P$ bound disks with disjoint interiors. 
        \item Each circle $c_v$ is centered at the vertex $v$. 
        \item $c_v,c_u\in P$ are tangent if and only if $\langle v,u\rangle$ is an edge of $V$.
        \item Arcs of mutually tangent circles $c_v,c_u,c_w\in P$ form a curvilinear triangle in $\Omega(\Gamma)$ if and only if $\langle u,v,w\rangle$ forms a positively oriented face of $V$.
        \item Any compact subset of $\Omega(\Gamma)$ intersects finitely many circles in $P$.
    \end{enumerate}
    A framed convex cocompact hyperbolic $3$-manifold $(M,f)$ \textbf{admits a circle packing} if a circle packing can be constructed on $\partial_\infty(M)$ with respect to the Kleinian group $\Phi^{-1}((M,f))$.
\end{defn}

\begin{figure}
    \begin{center}
        \includegraphics[scale=0.30]{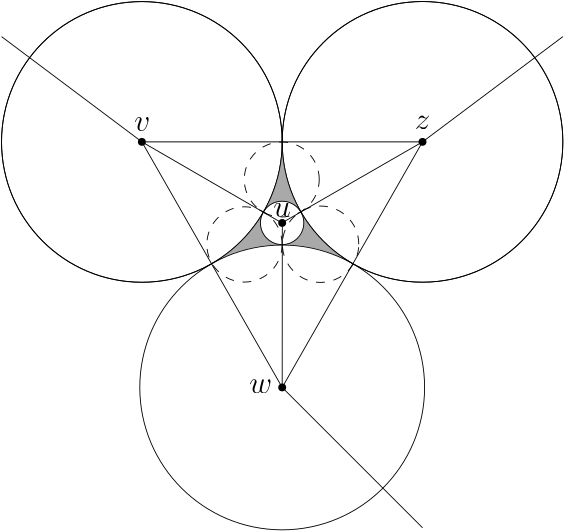}
    \end{center}
    \captionsetup{width=.8\linewidth}
    \caption{This depicts a portion of a circle packing $P$ with its associated triangulation. The shaded regions are curvilinear triangles in $\partial_{\infty}(\mathbb{H}^3/\Gamma)$. The dotted circles are the dual circles in $P^*$, including $c^{(v,u,w)}$.}
    \label{CirclePack}
\end{figure}

To elucidate the fourth condition in Definition \ref{def-circlepack}, Figure \ref{CirclePack} depicts a portion of a circle packing, in which curvilinear triangles (such as that bounded by the circles $c_v,c_u$, and $c_w$) are shown shaded. Note that the face $\langle v,z,w\rangle$ in the associated triangulation is not positively oriented, hence these circles do not bound a curvilinear triangle.

Fix a convex cocompact hyperbolic $3$-manifold $M=\mathbb{H}^3/\Gamma$ and a circle packing $P$ of $\partial_\infty(M)$ with respect to $\Gamma$. For circles $c_v,c_u,c_w$ in $P$ forming a curvilinear triangle in $\Omega(\Gamma)$, there is a unique \textbf{dual} circle $c^{(v,u,w)}$ orthogonal to each of $c_v,c_u$, and $c_w$ and intersecting them at the points of tangency. The dual circles, including $c^{(v,u,w)}$, are depicted dotted in Figure \ref{CirclePack}. The collection of all such dual circles corresponding to curvilinear triangles will be denoted $P^*$ and called the dual circle packing of $P$. Note that $P^*$ may not be an actual circle packing in the sense defined above: for example, the dual graph to a triangulation is not necessarily a triangulation.\par

For each circle $c$ in $P$ or $P^*$, let $H(c)\subseteq\mathbb{H}^3$ be the open hyperbolic half space meeting $S^2_\infty$ at the interior of $c$. We then define the \textbf{scooped manifold} \[M_P=\left(\mathbb{H}^3\;\backslash\bigcup_{c\in P,P^*}H(c)\right)\mathlarger{\mathlarger{\mathlarger{/}}}\Gamma.\]\par

The following lemma, proved by Fuchs, Purcell and Stewart in \cite[Prop 3.3]{PUR}, will tell us that that the boundary of a scooped manifold $M_P$ consists of hyperbolic ideal polyhedra. The faces of the polyhedra descend from the boundaries of $H(c)\subseteq\mathbb{H}^3$ for $c\in P\cup P^*$ and their edges descend from intersections of $H(c)$ and $H(c^*)$ for $c\in P$ and $c^*\in P^*$. Faces descending from boundaries of $H(c)$ for $c\in P$ will be colored white and faces coming from boundaries of $H(c^*)$ for $c^*\in P^*$ will be colored black. \par

\begin{lemma}
\label{scoopbound}
    Let $(M,f)\in\mathcal{H}_\infty$ be convex cocompact such that $\partial_\infty M$ admits a circle packing. Then, the scooped manifold $M_P$ has the following properties:
    \begin{enumerate}
        \item With the coloring above, no two faces of $\partial M_P$ of the same color share an edge. 
        \item The faces consist of totally geodesic ideal polygons, where the black faces are ideal triangles.
        \item The dihedral angle between faces is $\pi/2$.
    \end{enumerate}
\end{lemma}

Using the circle packing machinery developed by Brooks \cite{BRO} combined with the Density Theorem \cite{NAM}, Fuchs, Purcell and Stewart also prove the following result \cite{PUR}, which tells us that we can approximate any geometrically finite framed hyperbolic $3$-manifold by convex cocompact framed manifolds whose boundaries admit a circle packing.

\begin{lemma}
\label{packapprox}
    For any geometrically finite $(M,f)\in\mathcal{H}_\infty$ and $\varepsilon,R>0$, there exists a convex cocompact $(N,h_q)\in\mathcal{H}_\infty$ that is $(\varepsilon, R)$-close to $(M,f)$ and admits a circle packing $P$ with $B_N(q,R)$ contained in the scooped manifold $N_P$.
\end{lemma}

We will now utilize Lemmas \ref{scoopbound} and \ref{packapprox} to give a modification of Lemma \ref{packapprox} in which we approximate any framed hyperbolic $3$-manifold by ones whose convex core boundaries are given by a disjoint union of totally geodesic thrice punctured spheres. This construction is a modification of that carried out by Fuchs, Purcell, and Stewart in \cite[Construction 3.10]{PUR}, with  different goals.

\begin{prop}
\label{approxwithsph}
    For any $(M,f)\in\mathcal{H}_\infty$ and $\varepsilon,R>0$, there exists a geometrically finite $(N,h_q)\in \mathcal{H}_\infty$ that is $(\varepsilon,R)$-close to $(M,f)$ such that $\partial\CC(N)$ is a union of totally geodesic thrice punctured spheres, with $B_N(q,R)\subseteq \CC(N)$.
\end{prop}

\begin{proof}
    Lemma \ref{limofgeomfin} tells us that it will suffice to assume $M$ is geometrically finite. As in Lemma \ref{packapprox}, choose a convex cocompact $(N',h_{q'}')\in\mathcal{H}_\infty$ such that $(N',h_{q'}')$ is $(\varepsilon,R)$-close to $(M,f)$ and admits a circle packing $P$ with $B_{N'}(q',R)$ contained in the scooped manifold $N'_{P}$. \par

    Let $\overline{N'_P}$ be a copy of $N'_P$, with the reversed orientation. Form $N''$ by identifying each white face of $N'_P$ with its copy in $\overline{N'_P}$, via the identity map between these faces. Observe that $N''$ is a hyperbolic $3$-manifold with boundary given by a disjoint union of totally geodesic thrice punctured spheres: indeed, when mirrored white faces are identified, the edges of  black triangles are glued with their mirrored copies, with angles along the glued edges adding up to $\pi$, by Lemma \ref{scoopbound}. Note that $\mathrm{vol}(N_P')<\infty$, so $\mathrm{vol}(N'')=2\mathrm{vol}(N_P')<\infty$. \par

    To the hyperbolic $3$-manifold with totally geodesic boundary $N''$, there is an associated complete geometrically finite hyperbolic $3$-manifold (without boundary) $N$, given by attaching a Fuchsian end to each boundary component of $N''$; in particular, $\CC(N)$ is isometric to $N''$, so $\CC(N)$ indeed has boundary given by a union of totally geodesic thrice punctured spheres.\par

    Finally, let $h_q\in N$ denote the image of $h_{q'}'$ under the identifications and inclusion of $N''$ into $N$. Since $B_{N'}(q',R)\subseteq N_P'$ we have \[B_{N'}(q',R)\subseteq N_P'\hookrightarrow N''\hookrightarrow N\] where each of the inclusions is by isometry, so Lemma \ref{lemma-eRTransitivity} tells us that $(N,h_q)$ is $(\varepsilon,R)$-close to $(M,p)$, since $(N',h_{q'}')$ is.
\end{proof}

 Let $\mathcal{P}$ be the set of isometry classes of finite volume hyperbolic $3$-manifolds with non-empty boundary that consists of a disjoint union of totally geodesic thrice punctured spheres. Equivalently, $\mathcal{P}$ is the collection of convex cores of geometrically finite hyperbolic $3$-manifolds $N$ such that $\partial \CC(N)$ is a disjoint union of thrice punctured spheres. Note that $\mathcal{P}$ is countable, since each element corresponds to a compact $3$-manifold $M$, together with a pants decomposition of $\partial M$, of which there are only countably many.\par

We can now complete the construction of our dense leaf.

\begin{theorem}
\label{thm-denseleaf}
    There exists a hyperbolic $3$-manifold $M$ such that the leaf $\ell(M)$ is dense in $\mathcal{H}_\infty$.
\end{theorem}

\begin{proof}

    Since $\mathcal{P}$ is countable, we may enumerate $\mathcal{P}=\{M_0,M_1,M_2,...\}$. Note that each $M_i\in\mathcal{P}$ has an even number of boundary components, since the punctures of the thrice punctured sphere boundary components pair up into rank-1 cusps in $M_i$. In particular, for each $i\geq0$ we can select distinct components $L_i$ and $R_i$ of $\partial M_i$. All totally geodesic thrice punctured spheres are isometric, so we may construct a new manifold $M'$ by gluing $M_i$ to $M_{i+1}$ by an orientation reversing isometry taking $R_i$ to $L_{i+1}$ for all $i\geq 0$. See Figure \ref{gluingM} for a depiction of the gluing. \par

    \begin{figure}
    \begin{center}
        \includegraphics[scale=0.3]{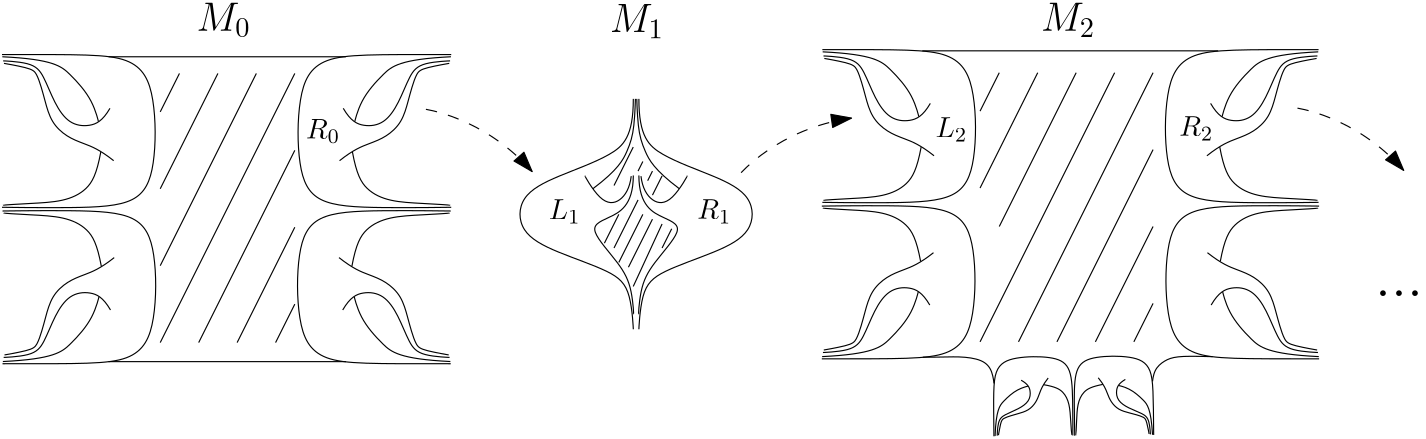}
    \end{center}
    \captionsetup{width=.8\linewidth}
    \caption{Some possible structures for the $\partial M_i$'s are depicted, along with the first few gluings in the construction of $M'$.}
    \label{gluingM}
\end{figure}

    Observe that $M'$ is a hyperbolic $3$-manifold with boundary given by a union of thrice punctured spheres. We now construct $M$ by gluing a Fuchsian end with totally geodesic thrice punctured sphere boundary to each component of the boundary of $M'$, again via an orientation reversing isometry from boundary to boundary. Once we check the completeness of $M$, it will follow that $M$ is the hyperbolic $3$-manifold such that $\CC(M)$ is isometric to $M'$. Note that $M$ is orientable. \par

    To check the completeness of $M$, it will suffice to show that the cusp neighborhoods in $M$ are foliated by Euclidean tori or annuli (see, for example, \cite[Section 3.10]{THUR}). For each piece in the gluing, either one of the $M_i$ or a Fuchsian end, any cusp in that piece is foliated by Euclidean tori or annuli. Any such annulus $S^1\times I$ intersects the boundary of that piece orthogonally, and the annulus is uniquely defined within that cusp neighborhood by its ``circumference," the minimal Euclidean length in the $S^1$ direction. In each gluing, each annulus is glued to another annulus with the same circumference, hence these annuli piece together to form Euclidean annuli or tori, which confirms that $M$ is complete.\par

    Let us now check that $\ell(M)$ is dense in $\mathcal{H}_\infty$. Choose $\varepsilon,R>0$ and $(N',h')\in\mathcal{H}_\infty$. Proposition \ref{approxwithsph} tells us that there is some $(N,h_q)\in\mathcal{H}_\infty$ such that $ \CC(N)\in\mathcal{P}$, $B_{N}(q,R)\subseteq \CC(N)$, and $(N,h_q)$ is $(\varepsilon,R)$-close to $(N',h')$. By construction, $\CC(N)$ embeds isometrically into $M$. Let $f\in \mathcal{F}M$ be the image of $h_q$ under this isometric embedding. Then, Lemma \ref{lemma-eRTransitivity} tells us that $(M,f)$ is $(\varepsilon,R)$-close to $(N',h')$, since $(N,h_q)$ is. Hence, we have shown that $\ell(M)$ is a dense leaf of $\mathcal{H}_\infty$.
\end{proof}

By permuting the enumeration of $\mathcal{P}$ in the proof of Theorem \ref{thm-denseleaf}, one produces uncountably many distinct leaves that are dense in $\mathcal{H}_\infty$.

%%%%%%%%%%%%%%%%%%%%%%%%%%%%%%%%%%%%%%%%%%%%%%%%%%%%%%%%%%%%%%%%%%%%%%%%%%%%%%%%%%%%%%%%%%%%%%%%%%%%%%%

\section{Paths}
\label{sec-paths}

This section will examine paths in $\mathcal{H}$, starting with Section \ref{subsec-pathexamples}, in which we will present a number of examples. In addition to general interest, these examples will motivate some of our theorems and demonstrate sharpness of some hypotheses. Section \ref{subsec-PathConstraints} will develop the general theory that we will use to understand constraints on possible paths in $\mathcal{H}$. Section \ref{subsec-pathcomps} will conclude with constructing an infinite family of path components of $\mathcal{H}_\infty$ (Theorem \ref{Thm-GGluePathComp}), in particular demonstrating that $\mathcal{H}_\infty$ is not path connected (Corollary \ref{cor-notpathconn}).

\subsection{Examples}
\label{subsec-pathexamples}

We have already encountered the first example, but we will briefly record it here.\par

\vspace{.1in}

\hypertarget{Ex-MovingFrame}{\textbf{Example A.1: (Moving frame)}}
For a fixed hyperbolic $3$-manifold $M$ and path $\gamma:[0,1]\rightarrow\mathcal{F}M$, the map $t\mapsto(M,\gamma(t))$ is a path in $\mathcal{H}$. The fact that this describes a path follows directly from Proposition \ref{pathprop}, which establishes the continuity of the map $L_M:\mathcal{F}M\rightarrow\mathcal{H}$. Since a change of baseframe corresponds to conjugating the corresponding Kleinian group, this example corresponds to fixing a Kleinian group $\Gamma$ and a path $t\mapsto\psi_t\in\mathrm{PSL}_2\C$ and taking the path 
in the geometric topology on $\mathcal{D}$ given by $t\mapsto\psi_t\Gamma\psi_t^{-1}$.$\qed$\par

\vspace{0.1in}

\hypertarget{Ex-FrameToInfty}{\textbf{Example A.2: (Frame to infinity)}} Suppose that for a fixed hyperbolic $3$-manifold $M$ and smooth path $\gamma:[0,\infty)\rightarrow M$, the injectivity radius of $M$ along $\gamma$ satisfies $\mathrm{inj}_{M}(\gamma(t))\rightarrow\infty$ as $t\rightarrow\infty$. Then, as in Lemma \ref{pathcomp}, if $f\in F_{\gamma(0)}M$ and $f^t$ is the time $t$ parallel transport of $f$ along $\gamma$, then $(M,f^t)$ is a path in $\mathcal{H}$ that converges to $(\mathbb{H}^3,\mathcal{O}_O)$.$\qed$\par

Example \hyperlink{Ex-FrameToInfty}{A.2} is an initial example of the topology of the underlying manifold changing along a path in $\mathcal{H}$. As seen in Lemma \ref{pathcomp}, a path such as $\gamma$ may be constructed in $M=\mathbb{H}^3/\Gamma$ for any Kleinian group $\Gamma$ with $\Omega(\Gamma)\neq\varnothing$.

\vspace{0.1in}

\hypertarget{Ex-QCDef}{\textbf{Example B.1: (Quasi-conformal deformations)}} \textit{Suppose $\Gamma$ is a nonelementary geometrically finite Kleinian group. For any path of geometrically finite representations $G:[0,1]\rightarrow D_{\GF}(\Gamma)$ (see Section \ref{Sec-Reps} for definitions), the map $t\mapsto G(t)(\Gamma)$ is a path in $\mathcal{D}$.}\par

\vspace{0.05in}

Define $G_{im}:[0,1]\rightarrow\mathcal{D}$ by $G_{im}(t)=(G(t))(\Gamma)$. The proof of Marden's stability theorem \cite{MAR} demonstrates that the Dirichlet fundamental polyhedra with respect to the fixed basepoint $O\in\mathbb{H}^3$ for $G_{im}(t)$ vary uniformly continuously on compact balls about $O\in\mathbb{H}^3$ in the Hausdorff topology. In particular, $G_{im}(t)$ is a path of Kleinian groups with respect to the topology of polyhedral convergence (see \cite{MAT}), hence is a path of Kleinian groups in the geometric topology on $\mathcal{D}$. In other words, $G$ may also be viewed as a path into the space $S(\Gamma)$ of discrete and faithful representation of $\Gamma$ into $\mathrm{PSL}_2\C$, equipped with the topology of strong convergence. We also note that the continuity of $G_{im}$ with respect to the geometric topology follows immediately from a result of Anderson and Canary on strong convergence \cite{ANCAN}, since each $G(t)$ is type preserving. Following from the correspondence between quasiconformal deformations and quasi-isometries \cite{RENORM}, the path in $\mathcal{H}$ corresponding to $G_{im}$ is a path of quasi-isometric framed hyperbolic $3$-manifolds, where cusps map to cusps. $\qed$

\vspace{0.1in}

\hypertarget{Ex-CuspRank}{\textbf{Example B.2: (Cusp rank)}} \textit{Let $\Gamma$ be a geometrically finite Kleinian group containing a parabolic element $\psi\in\Gamma$ such that $\mathbb{H}^3/\Gamma$ has a rank-$1$ cusp corresponding to $\psi$; more precisely, $\langle\psi\rangle$ is a maximal abelian subgroup of $\Gamma$. Then, there exists a path $\varphi:[t_0,\infty)\rightarrow\mathrm{PSL}_2\C$ of parabolic isometries such that $\langle\psi,\varphi(t)\rangle$ is a rank-$2$ abelian group for all $t$ and $t\mapsto\langle\Gamma,\varphi(t)\rangle$ is a path in $\mathcal{D}$ converging to $\Gamma$ as $t\rightarrow\infty$.}\par

\vspace{0.05in}

This example is discussed  by Fuchs, Purcell and Stewart in \cite[Section 4]{PUR}, and we will simply expand on some details of their analysis, with a focus on the topology of the approximating manifolds. By conjugating $\Gamma$, we may assume $\psi$ corresponds to the M\"{o}bius transformation $z\mapsto z+1$. For $t>0$, let $\varphi_t$ denote the element of $\mathrm{PSL}_2\C$ corresponding to $z\mapsto z+it$. Using the upper half space model for $\mathbb{H}^3$, define the subsets \[B_{\pm}=B_{\pm}(t)=\{(z,h)\in\C\times\mathbb{R}_+=\mathbb{H}^3\;|\;\pm\mathrm{Im}(z)\geq t/2\}.\] An analysis (see \cite{PUR}) of the \textit{standard cusp regions} corresponding to $\psi$ as defined in \cite{BOWGF} tells us that for $t$ sufficiently large $B_+,B_-$, and $B_+\cup B_-$ are each precisely invariant for $\langle \psi\rangle$ in $\Gamma$. For such $t$, since $B_+=\overline{\mathbb{H}^3\backslash \varphi_t(B_-)}$, the second Klein-Maskit combination theorem (see \cite{MAT}, \cite{MASK}, \cite{ABI}) tells us that the group $\Gamma_t=\langle\Gamma,\varphi_t\rangle$ is discrete. Fix $t_0>0$ so that $t$ is ``sufficiently large" for all $t\geq t_0$: $\Gamma_t\in\mathcal{D}$ for all $t\geq t_0$. That $t\mapsto\Gamma_t$ describes a path in $\mathcal{D}$ for $t\in[t_0,\infty)$ follows from Example \hyperlink{Ex-QCDef}{B.1}. \par

Before we prove that $\Gamma_t\rightarrow\Gamma$ as $t\rightarrow\infty$, let us discuss how to construct $M_t=\mathbb{H}^3/\Gamma_t$ from $M=\mathbb{H}^3/\Gamma$. Let $\pi:\mathbb{H}^3\rightarrow M$ denote the projection map corresponding to $\Gamma$. One constructs $M_t$ by gluing the two boundary components of $M\backslash(\pi(B_+)\cup\pi(B_-))$. In the language of pared manifolds (see \cite[Section 7]{THUI}), $M$ is homeomorphic to the interior of a pared closed $3$-manifold $(\overline{M},P)$, and the rank-$1$ cusp associated to $\psi$ corresponds to a component of $P$ which is an incompressible annulus $A$ contained in a surface $S\subseteq\partial\overline{M}$. Let $c\subseteq S$ denote a curve representing the homotopy class of a core curve of $A$. The surface $S\subseteq\partial\overline{M}$ has a neighborhood in $\overline{M}$ parameterized by $S\times[0,1]$, where $S=S\times\{1\}\subseteq\partial\overline{M}$. Let $\overline{M}'$ denote the closed $3$-manifold given by drilling out an open tubular neighborhood of $c\times\{1/2\}$; assume this neighborhood is chosen to be sufficiently small to be contained in $S\times[1/4,3/4]$, and let $T$ be the toroidal boundary of this neighborhood. Set $P'=(P\backslash A)\cup T\subseteq\partial\overline{M}'$. A cartoon of the local picture of the construction of $(\overline{M}',P')$ from $(\overline{M},P)$ is given in Figure \ref{Figure-CuspRankPlus1}. Then, the pared closed $3$-manifold corresponding to each $M_t$ is $(\overline{M}',P')$.\par

\begin{figure}
    \begin{center}
        \includegraphics[scale=0.35]{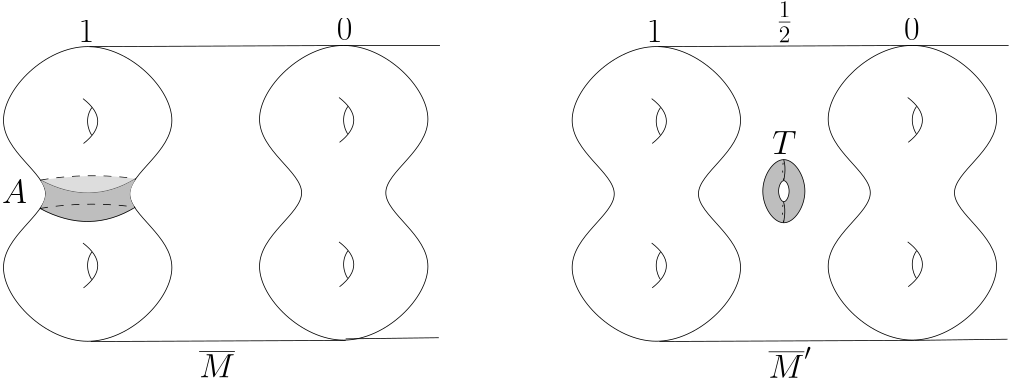}
    \end{center}
    \captionsetup{width=.8\linewidth}
    \caption{This depicts the construction of $(\overline{M}',P')$ from $(\overline{M},P)$ in Example B.2, depicting $S$ as a genus-$2$ surface and the paring as shaded.}
    \label{Figure-CuspRankPlus1}
\end{figure}

Let us confirm that $\Gamma_t\rightarrow\Gamma$ as $t\rightarrow\infty$. Note that for large enough $t$, $O\not\in B_+(t)\cup B_-(t)$, where ${\mathcal{O}}_O$ is our fixed baseframing for $\mathbb{H}^3$. From the above construction of $M_t$ from $M$, we can then see that for any $R>0$ and any $t$ sufficiently large, the inclusion map $B_{\mathbb{H}^3}(O,R)\hookrightarrow\mathbb{H}^3$ descends to an embedding $B_{M}(\pi(O),R)\hookrightarrow M_t$. Thus, Definition \ref{eRclosedef2} directly tells us that $\Gamma_t\rightarrow\Gamma$ as $t\rightarrow\infty$. $\qed$ \par

\vspace{0.1in}

\noindent The following example appears as Theorem \ref{introthm-NonTameLimit} in the introduction.

\stepcounter{theorem}
\hypertarget{Ex-NonTame}{\textbf{Example B.3: (Non-tame limit)}} \textit{For a closed surface $S$ of genus at least two, let $C\subseteq S\times\Z\subseteq S\times \R$ be a collection of disjoint simple closed curves so that no two curves in $C$ are isotopic in $S\times \R$ and the curves $C_n=C\cap(S\times\{n\})$ give a pants decomposition of the surface $S_n=S\times\{n\}$ for all $n\in \Z$. Then, there exists a path $G:[1,\infty]\rightarrow\mathcal{D}$ so that $G(t)$ is geometrically finite for all $t<\infty$ and $\mathbb{H}^3/G(\infty)$ is isometric to the unique hyperbolic structure on $(S\times \R)\backslash C$.}\par

\vspace{0.05in}

For the construction, it will be slightly easier to replace the assumption that $C_0$ is a pants decomposition of $S_0$ with the assumption that $C_0=\varnothing$; this simply amounts to a reparameterization of $S\times\R$ and $C$, so nothing is lost. For each $n$, let $P_n\subseteq S_n$ denote a tubular neighborhood of $C_n$ as a submanifold of $S_n$, so that $P_n$ is a disjoint union of annuli whose core curves are the elements of $C_n$. It follows, for example by \cite[Theorem 6.4]{THUII}, that there is a hyperbolic structure, call it $M_1$, on the interior of the pared compact $3$-manifold $(S\times[-1,1],P_{-1}\cup P_{1})$ such that the annuli in the paring correspond to rank-$1$ cusps in $M_1$. Note that since $C_{-1}$ and $C_{1}$ are pants decompositions of $S_{-1}$ and $S_1$, respectively, $\partial \CC(M_1)$ must be a disjoint union of totally geodesic thrice punctured spheres, and hence the hyperbolic structure on $M_1$ is uniquely determined. Fix a Kleinian group $\Gamma_1\in\mathcal{D}$ so that $M_1=\mathbb{H}^3/\Gamma_1$. \par

Let $T\subseteq S\times\R$ be a small closed tubular neighborhood of $C$ so that $\partial T$ is a disjoint union of tori and $(S\times\R)\backslash C$ is homeomorphic to $(S\times \R)\backslash T$. Additionally, let $T_n$ denote the union of the components of $T$ corresponding to curves in $C_n$, and assume $T_n\subseteq S\times[n-(1/4),n+(1/4)].$\par

For $n\geq1$, assume that we have inductively constructed a Kleinian group $\Gamma_n$ so that $\Gamma_{n-1}\leq\Gamma_n$ (where $\Gamma_0=\{\mathds{1}\}$) and $M_n=\mathbb{H}^3/\Gamma_n$ gives a hyperbolic structure on interior of the pared $3$-manifold 
\begin{equation}\label{eq:paring}
    \overline{M_n}=(S\times[-n,n]\backslash\textstyle\bigcup_{-n< j<n} \mathrm{int}(T_j),P_{-n}\cup P_n\cup(\textstyle\bigcup_{-n<j<n}\partial T_j)),
    \tag{$\star$}
\end{equation}
 relative to the paring. The hyperbolic structure on $M_n$ is unique, since $\partial \CC(M_n)$ is a disjoint union of totally geodesic thrice punctured spheres. For each rank-$1$ cusp of $M_n$, Example \hyperlink{Ex-CuspRank}{B.2} allows us to bring in a parabolic isometry from infinity to construct a path in $\mathcal{D}$ from $\Gamma_n$ to $\Gamma_n'$ so that $\Gamma_n$ is contained in all intermediate groups, including $\Gamma_n'$, and $M_n'=\mathbb{H}^3/\Gamma_n'$ is homeomorphic to the interior of $S\times[-(n+1),n+1]\backslash(\bigcup_{-n\leq j\leq n} T_n)$. Note that $\Gamma_n'$ is not uniquely defined: we bring in each new parabolic isometry along a path from infinity in $\mathrm{PSL}_2\C$, but there is choice in ``how far" to bring each in along their respective paths. Note, however, that the rigidity of $M_n$ implies that $\CC(M_n)$ embeds isometrically into $M_n'$ under the natural quotient map $M_n\rightarrow M_n'$ arising from the inclusion $\Gamma_n\leq\Gamma_n'$.\par

By work of Ahlfors, Bers, Marden, and Sullivan (see eg. \cite{MAT}) we may identify the space of conjugacy classes of geometrically finite representations $\GF(\Gamma_n')$ with the Teichm\"{u}ller space $\mathcal{T}(S_{-(n+1)})\times \mathcal{T}(S_{n+1})$. The pants decompositions of $S_{-(n+1)}$ and $S_{n+1}$ given by $C_{-(n+1)}$ and $C_{n+1}$ respectively provide Fenchel-Nielsen coordinates \cite{BEN} for the respective Teichm\"{u}ller spaces. Consider a path $t\mapsto[\rho_t]\in \GF(\Gamma_n')$ corresponding to sending the lengths of the curves in $C_{-(n+1)}$ and $C_{n+1}$ to zero as $t\rightarrow\infty$. There is a unique lift to a path $t\mapsto\rho_t\in D_{\GF}(\Gamma_n')$ so that $\rho_t|_{\Gamma_n}$ is the identity map for all $t$, by Lemma \ref{Lemma-AHNonelemSubgroupConvg} and the rigidity of $\CC(M_n)$. By Example \hyperlink{Ex-QCDef}{B.1}, the map $t\mapsto\rho_t(\Gamma_n')\in\mathcal{D}$ is a path in $\mathcal{D}$.\par

Observe that the hypothesis that all curves in $C$ are non-isotopic guarantees that as a pared $3$-manifold, $\overline{M_{n+1}}$ is acylindrical. In the language of \cite{EGG}, this implies that the gallimaufry associated to the paring on $\overline{M_{n+1}}$ is doubly incompressible. In particular, it then follows from \cite{EGG} and Lemma \ref{Lemma-AHNonelemSubgroupConvg} that as $t\rightarrow\infty$, the path $t\mapsto\rho_t(\Gamma_n')$ converges to a Kleinian group $\Gamma_{n+1}$ so that $M_{n+1}=\mathbb{H}^3/\Gamma_{n+1}$ gives a hyperbolic structure on the interior of the pared hyperbolic $3$-manifold $\overline{M_{n+1}}$, as in (\ref{eq:paring}). Noting that $\Gamma_{n}\leq\Gamma_{n+1}$, the inductive step is complete.\par

Hence, we may concatenate and reparameterize the above paths to define a path $G:[1,\infty)\rightarrow\mathcal{D}$ so that $G(n)=\Gamma_n$ for all $n$ and $\Gamma_n\leq G(t)$ for all $t\geq n$. Additionally, observe that for all $t\geq n$ the natural projection $M_n\rightarrow \mathbb{H}^3/G(t)$ embeds $\CC(M_n)$ isometrically into $\mathbb{H}^3/G(t)$. By Lemma \ref{lemma-eRTransitivity}, the path $G(t)$ converges to some $\Gamma_\infty$ where $M_\infty=\mathbb{H}^3/\Gamma_\infty$ is homeomorphic to $(S\times\R)\backslash C$. $\qed$\par

An example of a collection of curves such as those in Example \hyperlink{Ex-NonTame}{B.3} may be constructed by picking a pants decomposition $D\subseteq S$ of $S$, a pseudo-Anosov homeomorphism $f:S\rightarrow S$, and letting $C=\bigcup(f^n(D)\times\{n\})$. In this case, the hyperbolic structure on $M_\infty$, which is homeomorphic to $(S\times\R)\backslash C$, has a $\Z$-symmetry. In particular, each piece $(S\times[n,n+1])\backslash C$ is isometric.\par

\vspace{0.1in}

\hypertarget{Ex-Schottky}{\textbf{Example C.1: (Schottky)}} In this example, we will construct a path of rank-2 Schottky (free) groups that converges to an infinite cyclic group. For $s>1$, define \[\psi_s=\begin{bmatrix}
s & 0 \\
s & 1/s 
\end{bmatrix}\in\mathrm{PSL}_2\R.\] As discussed in \cite{FORD}, if \[C_s=\{z\in\C \;|\; |z+1/s^2|=1/s\}\;\;\text{ and }\;\;C_s'=\{z\in\C \;|\; |z-1|=1/s\},\] then $\psi_s$ acts on the Riemann sphere $S^2_\infty$ by mapping the exterior of $C_s$ onto the interior of $C_s'$.\par

Let $D$ and $D'$ be two circles on the Riemann sphere whose designated interiors are disjoint from each other and disjoint from $\{z\in\C:|z|\leq 2\}$ (and hence disjoint from $C_s,C_s'$ for all $s>1$). Now, choose $\varphi\in\mathrm{PSL}_2\C$ to be any isometry taking the exterior of $D$ to the interior of $D'$.\par

For $t\in(1,\infty]$, let $\Gamma_t\leq\mathrm{PSL}_2\C$ be the subgroup given by

\[ \Gamma_t = 
    \begin{cases} 
      \langle\varphi,\psi_t\rangle, & 1<t<\infty \\
      \langle\varphi\rangle, & t=\infty .
   \end{cases}
\]
It is classical, following from the first Klein-Maskit combination theorem (\cite{MAT}, \cite{MASK}, \cite{ABI}), that each $\Gamma_t$ is a discrete free group. That $t\mapsto\Gamma_t$ is a path in $\mathcal{D}$ on the domain $t\in(1,\infty)$ follows from Example \hyperlink{Ex-QCDef}{B.1}. To see that $\Gamma_t\rightarrow\Gamma_\infty$ as $t\rightarrow\infty$, note that each word $w$ in the abstract free group $\langle a,b\rangle$ corresponds to a path $t\mapsto w_t$ in $\mathrm{PSL}_2\C$ with domain $t\in(0,\infty)$ under the isomorphism $\langle a,b\rangle\rightarrow \Gamma_t$ given by $a\mapsto\varphi$, $b\mapsto\psi_t$. For any reduced word $w\in\langle a,b\rangle\backslash \langle a\rangle$, the path $w_t$ must diverge to $\infty$ as $t\rightarrow\infty$; for example, note that a loop in $\mathbb{H}^3/\Gamma_t$ representing the conjugacy class of $w_t\in\Gamma_t$ must have length at least that of the closed geodesic corresponding to $\psi_t$. Hence, the only elements of $\mathrm{PSL}_2\C$ that may be accumulated on by the $\Gamma_t$'s as $t\rightarrow\infty$ are the elements of $\langle \varphi\rangle$, so it follows from Fact \ref{ChabProps} that $\Gamma_t\rightarrow\Gamma_\infty$ as $t\rightarrow \infty$.$\qed$\par

\vspace{0.1in}

\hypertarget{Ex-CCProd}{\textbf{Example C.2: (Convex cocompact free product)}} \textit{For Kleinian groups $G$ and $H$ with $\Omega(G),\Omega(H)\neq\varnothing$, there exists a path $\varphi:[0,\infty)\rightarrow\mathrm{PSL}_2\C$ so that if $H_t=\varphi(t)H\varphi(t)^{-1}$, then the following hold:
\begin{enumerate}
    \item each $\langle G,H_t\rangle$ is discrete and splits as a free product $\langle G,H_t\rangle =G*H_t$,
    \item $t\mapsto G*H_t$ is a path in $\mathcal{D}$,
    \item $G*H_t\rightarrow G$ in $\mathcal{D}$ as $t\rightarrow\infty$.
\end{enumerate}}

\vspace{0.05in}

We may choose an open disk $D_1\subseteq\Omega(G)$ such that $D_1$ is precisely invariant under $\{\mathds{1}\}$ in $G$. Fix $x_1,x_2\in D_1$. By replacing $H$ with a $\mathrm{PSL}_2\C$-conjugate subgroup, we may assume that $x_1\in\Omega(H)$. Now, we may choose an open disk $D_2\subseteq\Omega(H)$ centered at $x_1$ such that $D_2$ is precisely invariant under $\{\mathds{1}\}$ in $H$.\par

For $t>0$, let $\varphi_t$ be the hyperbolic isometry with repelling fixed point $x_1$, attracting fixed point $x_2$, no torsion part, and hyperbolic translation length $t$ along the geodesic limiting to $x_1$ and $x_2$. For $t>0$, define \[\Gamma_t=\langle G,\varphi_t H \varphi_t^{-1}\rangle.\] For $t$ sufficiently large, say for $t\geq t^*$, we have $S^2_\infty\backslash\overline{D_1}\subseteq\varphi_t(D_2)$. Since $D_2$ is precisely invariant under $\{\mathds{1}\}$ in $H$, it follows that for $t\geq t^*$, $S^2_\infty\backslash\overline{D_1}$ is precisely invariant under $\{\mathds{1}\}$ in $\varphi_tH\varphi_t^{-1}$. Hence, since $D_1$ is precisely invariant under $\{\mathds{1}\}$ in $G$, the first Klein-Maskit combination theorem, as stated in \cite{ABI}, implies that $\Gamma_t$ is discrete for $t\geq t^*$, with $\Gamma_t\cong G * \varphi_t H \varphi_t^{-1}\cong G * H$. A similar argument as made in Example \hyperlink{Ex-Schottky}{C.1} shows that $\Gamma_t$ is a path in $\mathcal{D}$ on the domain $t\in[t^*,\infty)$ and that $\Gamma_t\rightarrow G$ as $t\rightarrow\infty$. \par

Let us also note what is occurring in the corresponding path of framed manifolds in $\mathcal{H}$. Let $N_1=\mathbb{H}^3/G$, $N_2=\mathbb{H}^3/H$, and $M_t=\mathbb{H}^3/\Gamma_t$, for $t\geq t^*$. Roughly, the convex core $\CC(M_t)$ contains isometrically embedded copies of $\CC(N_1)$ and $\CC(N_2)$ that are connected by a thin $1$-handle whose length goes to infinity as $t\rightarrow\infty$. If $f^t\in\mathcal{F}M_t$ is the $\Gamma_t$-projection of the frame $\mathcal{O}_O$, then the embedded $\CC(N_1)\hookrightarrow M_t$ stays fixed relative to $f^t$, while the embedded $\CC(N_2)\hookrightarrow M_t$ gets infinitely far away from $f^t$, disappearing and leaving behind $N_1$ as $t\rightarrow\infty$. $\qed$\par

\textbf{Remark:} There is a sense in which Examples \hyperlink{Ex-QCDef}{B.1} and \hyperlink{Ex-CCProd}{C.2} are representative of all paths of convex cocompact Kleinian groups. Specifically, if every group along a path in $\mathcal{D}$ is convex cocompact, the only way the isomorphism type can change is for free factors to come in from or diverge to infinity. This will be made precise in future work of the author.

%%%%%%%%%%%%%%%%%%%%%%%%%%%%%%%%%%%%%%%%%%%%%%%%%%%%%%%%%%%%%%%%%%%%%%%%%%%%%%%%%%%%%%%%%%%%%%%%%%%%%%%%%%%%%%%%%%%%%%%%%%%%%%%%%%%%%%%%%%%%%%%%%%%

\subsection{Path Machinery}
\label{subsec-PathConstraints}

This subsection will develop machinery for discussing paths in $\mathcal{H}$, and will establish some preliminary results on how these paths behave. We start with a lemma concerning sequences in $\mathcal{D}$.

\begin{lemma}
    \label{Unbhd}
    Fix $\Gamma\in\mathcal{D}$. For each $\psi\in\Gamma$, there exists a neighborhood $U_\psi\subseteq\mathrm{PSL}_2\C$ such that if  $\Gamma_n\rightarrow\Gamma$ is a convergent sequence in $\mathcal{D}$, then $|\Gamma_n\cap U_\psi|=1$ for $n$ sufficiently large.
\end{lemma}

\begin{proof}

    First, we will address the case $\psi=\mathds{1}$. Identify the Lie algebra $\mathfrak{psl}_2\C$ with the tangent space $T_{\mathds{1}}\mathrm{PSL}_2\C$ and fix a norm $\|\cdot\|$ on $\mathfrak{psl}_2\C$. For $r>0$, let $B_r=\{v\in\mathfrak{psl}_2\C\;|\;\|v\|<r\}$. Note that we may pick $r>0$ small enough so that the Riemannian exponential map $\mathrm{exp}:\mathfrak{psl}_2\C\rightarrow\mathrm{PSL}_2\C$ restricts to a diffeomorphism on $B_r$ (see eg. \cite{LEE}), and the discreteness of $\Gamma$ allows us to ensure $\overline{\mathrm{exp}(B_r)}\cap\Gamma=\{\mathds{1}\}$. Let $U=\mathrm{exp}(B_{r/2})$ and $U'=\mathrm{exp}(B_r)$, and observe that $\overline{U'\backslash U}$ is compact and $\overline{U'\backslash U}\cap\Gamma=\varnothing$.\par

    Now, suppose for contradiction that $\{\Gamma_n\}$ is a sequence in $\mathcal{D}$ converging to $\Gamma$ such that there exists $\psi_n\in(\Gamma_n\cap U)\backslash\{\mathds{1}\}$ for all $n$. For each $n$, define $v_n\in B_{r/2}$ such that $\mathrm{exp}(v_n)=\psi_n$, and choose $m_n\in\Z_+$ such that $r/2\leq\|m_n\cdot v_n\|<r$. Then, \[\psi_n^{m_n}=\mathrm{exp}(m_n\cdot v_n)\in\overline{U'\backslash U}\cap\Gamma_n.\] Since $\overline{U'\backslash U}$ is compact, the sequence $\{\psi_n^{m_n}\}_{n\geq1}$ has an accumulation point $\psi^*$ in $\overline{U'\backslash U}$, so the first condition for convergence in the geometric topology in Fact \ref{ChabProps} implies that $\psi^*\in\overline{U'\backslash U}\cap\Gamma$. This, however, is a contradiction, so taking $U_{\mathds{1}}=U$ suffices.\par

    To address case of general $\psi\in\Gamma$, first consider the map $F:\mathrm{PSL}_2\C\times\mathrm{PSL}_2\C\rightarrow\mathrm{PSL}_2\C$ given by $F(\varphi,\tau)=\varphi^{-1}\tau$. Note that $F$ is continuous, so $F^{-1}(U)$ is an open neighborhood of $\mathds{1}\times\mathds{1}$, where $U$ is the same subset as in the previous paragraphs. We can then find $U^*\subseteq U$ containing $\mathds{1}$ such that $U^*\times U^*\subseteq F^{-1}(U)$. Observe then that $\varphi^{-1}\tau\in U$ for all $\varphi,\tau\in U^*$. \par

    For $\psi\in\Gamma\backslash\{\mathds{1}\}$, set $U_\psi=\psi\cdot U^*$. As above, consider a sequence $\{\Gamma_n\}$ converging in $\mathcal{D}$ to $\Gamma$. If $\varphi_n,\tau_n\in\Gamma_n\cap U_\psi$ for some $n$, then $\psi^{-1}\varphi_n,\psi^{-1}\tau_n\in U^*$, so the choice of $U^*$ tells us \[\varphi_n^{-1}\tau_n=(\psi^{-1}\varphi_n)^{-1}(\psi^{-1}\tau_n)\in U=U_{\mathds{1}}.\] We know from above that for $n$ sufficiently large, we must have $\Gamma_n\cap U=\{\mathds{1}\}$, so it follows that $|\Gamma_n\cap U_\psi|\leq1$ for $n$ sufficiently large. Since $\Gamma_n\rightarrow\Gamma$, the second condition for convergence in the geometric topology given in Fact \ref{ChabProps} tells us that we in fact have $|\Gamma_n\cap U_\psi|=1$ for $n$ sufficiently large.
\end{proof}

\begin{defn}
\label{GammaPathDef}
    Let $\Gamma:[a,b]\rightarrow\mathcal{D}$ be a path. For $s\in [a,b]$ and $\psi\in\Gamma(s)$, we will call a path $j:I\rightarrow\mathrm{PSL}_2\C$ a $\Gamma$\textbf{-path through }$\psi$\textbf{ based at} $s$ if $j$ satisfies the following  conditions: 
\begin{enumerate}[i.]
    \item $I\subseteq[a,b]$ is an interval containing $s$
    \item $j(t)\in\Gamma(t)$ for all $t\in I$
    \item $j(s)=\psi$.
\end{enumerate}
\end{defn}

For the remainder of the subsection, fix a path $\Gamma:[0,1]\rightarrow\mathcal{D}$, and let $\Gamma_t=\Gamma(t)$. The following proposition gives local existence of $\Gamma$-paths through $\psi$ based at $s$ for any $s\in[0,1]$ and $\psi\in\Gamma_s$. This may be interpreted as a path analogue of the second condition for sequential convergence in the geometric topology on $\mathcal{D}$, as stated in Fact \ref{ChabProps}.

\begin{prop}
\label{pslpath}
    For all $s\in[0,1]$ and $\psi\in\Gamma_s$, there exists an interval open in $[0,1]$ that is the domain of a $\Gamma$-path through $\psi$ based at $s$. Further, if $j:I\rightarrow\mathrm{PSL}_2\C$ is any $\Gamma$-path through $\psi$ based at $s$, $j$ is the unique $\Gamma$-path through $\psi$ based at $s$ with domain $I$.
\end{prop}

\begin{proof}
     First, fix a neighborhood $U_\psi\subseteq\mathrm{PSL}_2\C$ of $\psi$ satisfying the conclusion of Lemma \ref{Unbhd}, with respect to sequences converging to $\Gamma_s$. Observe that for $t$ sufficiently close to $s$, we have $|\Gamma_t\cap U_\psi|=1$. Hence, for some open interval $I'$ containing $s$, we may define a function $j:I'\rightarrow\mathrm{PSL}_2\C$ by setting $j(t)$ to be the unique element in $\Gamma_t\cap U_\psi$. In particular, $j(s)=\psi$. The fact that $j$ is continuous follows from $\Gamma$ being a path in $\mathcal{D}$. Hence, $j$ is a $\Gamma$-path through $\psi$ based at $s$, and observe that $j$ is the uniquely defined such path on this domain. \par

     Now, fix an interval $I\subseteq[0,1]$ containing $s$ and suppose $j_1,j_2:I\rightarrow\mathrm{PSL}_2\C$ are $\Gamma$-paths through $\psi$ based at $s$. Let $t_*=\inf\{t\in I\;|\;j_1(t')=j_2(t')\text{ for all $t'\in[t,s]$}\}$. If $t_*\neq\inf I$, then continuity implies that $j_1(t_*)=j_2(t_*)$, and it follows that $j_1$ and $j_2$ are $\Gamma$-paths through $j_1(t_*)$ based at $t_*$. The preceding paragraph then implies that $j_1$ and $j_2$ must agree near $t_*$, contradicting the definition of $t_*$. Hence, $t_*=\inf I$, and with an analogous argument with a supremum, the uniqueness claim follows.
\end{proof}

Notice that Proposition \ref{pslpath} implies that for each $s\in[0,1]$ and $\psi\in\Gamma_s$, there is an interval $I_\psi^s\subseteq[0,1]$ that is the maximal domain on which we can define a $\Gamma$-path through $\psi$ based at $s$, call it $j_\psi^s:I_\psi^s\rightarrow\mathrm{PSL}_2\C$. For an interval $I\subseteq[0,1]$, we define the \textit{frontier} $\mathrm{Fr}(I)=\overline{I}\backslash\mathrm{int}_{[0,1]}(I)$, where $\mathrm{int}_{[0,1]}(I)$ is the interior of $I$ with respect to the topology on $[0,1]$. Note that $\mathrm{Fr}(I)\subseteq\{\inf I, \sup I\}$.

\begin{lemma}
    \label{Lemma-j-path-props}
    Fix $s\in[0,1]$ and $\psi\in\Gamma_s$.
    \begin{enumerate}
        \item The interval $I_\psi^s$ is open in the topology on $[0,1]$. In particular, if $\inf I_\psi^s\not\in\mathrm{Fr}(I_\psi^s)$, then $\inf I_\psi^s=0$; similarly, if $\sup I_\psi^s\not\in\mathrm{Fr}(I_\psi^s)$, then $\sup I_\psi^s=1$.
        \item As $t\in I_\psi^s$ converges to a point in $\mathrm{Fr}(I_\psi^s)$, $j_\psi^s(t)$ diverges to $\infty$ in $\mathrm{PSL}_2\C$.
    \end{enumerate}
\end{lemma}

\begin{proof}
    The first claim follows immediately from Proposition \ref{pslpath}, which tells us that $I_\psi^s$ is a union of intervals containing $s$ that are open in $[0,1]$.\par
    
    For the second claim, fix $t'\in\mathrm{Fr}(I_\psi^s)$ and suppose $j_\psi^s(t)$ does not diverge to $\infty$ as $t\in I_\psi^s$ converges to $t'$. Then, as $t$ converges to $t'$, $j_\psi^s(t)$ accumulates on some $\varphi\in\mathrm{PSL}_2\C$. By the first condition for convergence in the geometric topology on $\mathcal{D}$, we must then have $\varphi\in\Gamma_{t'}$. Let $U_\varphi\subseteq\mathrm{PSL}_2\C$ be the neighborhood of $\varphi$ given by Lemma \ref{Unbhd} with respect to $\Gamma_{t'}$, and note that it follows that $j_\psi^s(t)\in U_\varphi$ for $t$ sufficiently close to $t'$. The proof of Lemma \ref{Unbhd} shows that one can take the neighborhood $U_\varphi$ to be arbitrarily small, so we see that $j_\psi^s(t)$ converges to $\varphi\in\Gamma_{t'}$ as $t$ approaches $t'$. Hence, the maximality of $I_\psi^s$ implies $t'\in I_\psi^s$, which contradicts that $I_\psi^s$ is open.
\end{proof}

Let $\overline{\mathrm{PSL}_2\C}=\mathrm{PSL}_2\C\cup\{\infty\}$ denote the one-point compactification of $\mathrm{PSL}_2\C$. For $s,t\in[0,1]$, we now define a map $J_{s,t}:\Gamma_s\rightarrow\overline{\mathrm{PSL}_2\C}$ by

\begin{equation}
\label{eqn-JDef}
    J_{s,t}(\psi)=
\begin{cases} 
      j_\psi^s(t), & t\in I_\psi^s \\
      \infty, & \text{else}.
   \end{cases}
\end{equation}

For each $s\in[0,1]$ and $\psi\in\Gamma_s$, the map $t\mapsto J_{s,t}(\psi)=j_\psi^s(t)$ on $I_\psi^s$ describes a path in $\mathrm{PSL}_2\C$, which diverges to infinity at any points of $\mathrm{Fr}(I_\psi^s)$. Outside of $I_\psi^s$, $J_{s,t}(\psi)$ is constant at the $\infty\in\overline{\mathrm{PSL}_2\C}$. We will record some facts about each map $J_{s,t}$, whose image is contained in $\Gamma_t\cup\infty$.

\begin{lemma}
\label{Lemma-JFacts}
    Fix $s\in[0,1]$.
    \begin{enumerate}
        \item The family of functions $J_{s,t}$ varies continuously in $t$ with respect to the compact-open topology.
        \item If $t,t'\in I_\psi^s$, then $(J_{t,t'}\circ J_{s,t})(\psi)=J_{s,t'}(\psi)$.
    \end{enumerate}
\end{lemma}
\begin{proof}
    For the first claim, it suffices to show that for all $\psi\in\Gamma_s$, $t\mapsto J_{s,t}(\psi)$ is a path in $\overline{\mathrm{PSL}_2\C}$, since $\Gamma_s$ is discrete. By definition, $t\mapsto J_{s,t}(\psi)$ agrees with the path $j_\psi^s$ on $I_\psi^s$. Hence, it therefore remains to observe that Lemma \ref{Lemma-j-path-props} implies that as $t$ converges to a point in the frontier of $I_\psi^s$, $J_{s,t}(\psi)$ converges to $\infty$ in $\overline{\mathrm{PSL}_2\C}$.\par
    
    For the second claim, observe that as $t''$ ranges between $t$ and $t'$, $t''\mapsto j_\psi^s(t'')$ describes a $\Gamma$-path through $j_\psi^s(t)$ based at $t$, so the uniqueness of such paths given by Proposition \ref{pslpath} gives the desired claim.
\end{proof}

For a Kleinian group $\Gamma'$, we will say that a map $J:\Gamma'\rightarrow\overline{\mathrm{PSL}_2\C}$ is \textbf{finite valued} if the image $J(\Gamma')$ is contained in $\mathrm{PSL}_2\C$. Understanding when the maps $J_{s,t}$ are finite valued will be an important tool for analyzing the path $t\mapsto\Gamma_t$. In particular, we will see that if $J_{s,t}$ is finite valued for all $t\in I\subseteq[0,1]$, then $J_{s,t}$ is an injective homomorphism for all $t\in I$, so the map $t\mapsto J_{s,t}\in S(\Gamma_s)$ is a path in the topology of strong convergence on the space $S(\Gamma_s)$ of discrete and faithful representations $\Gamma_s\rightarrow\mathrm{PSL}_2\C$.

\begin{prop}
\label{Prop-InjConds}
    For any $s,t\in[0,1]$ and any subgroup $H\leq\Gamma_s$, the following are equivalent:
    \begin{enumerate}
        \item $J_{s,t}|_H$ is finite valued
        \item $J_{s,t}|_H$ is an injective homomorphism
        \item $t\in I_\psi^s$ for all $\psi\in H$.
    \end{enumerate}
\end{prop}

\begin{proof}
    
    ($3\Rightarrow 2$): Let $I_*=\bigcap_{\psi\in H}I_\psi^s$, and fix $\psi,\varphi\in H$. Notice that if $t\in I_*$, then $t\mapsto J_{s,t}(\psi)\cdot J_{s,t}(\varphi)$ is a $\Gamma$-path through $\psi\varphi$ based at $s$, so the uniqueness statement in Proposition \ref{pslpath} implies that $J_{s,t}(\psi)\cdot J_{s,t}(\varphi)=J_{s,t}(\psi\varphi)$ for all $t\in I_*$. Thus, $J_{s,t}|_H$ is a homomorphism for all $t\in I_*$.\par

    Now, fix $t\in I_*$ and suppose that $J_{s,t}(\psi)=\mathds{1}$ for some $\psi\in H$, and observe that the path $j_\psi^s$ is a $\Gamma$-path through $\mathds{1}$ based at $t$. The constant path at $\mathds{1}$ is the unique $\Gamma$-path through $\mathds{1}$ based at $t$, so we must have that $\psi=\mathds{1}$, and the injectivity of $J_{s,t}$ follows.\par

    The implications ($2\Rightarrow 1$) and ($1\Rightarrow 3$) are immediate from the definitions.
\end{proof}

Informally, the previous proposition says that $J_{s,t}|_H$ will be an injective homomorphism if and only if as $t'$ ranges from $s$ to $t$, the path $t'\mapsto J_{s,t'}(\psi)$ does not diverge to $\infty$ for any $\psi\in H$. Additionally, we see that for any distinct $\psi,\varphi\in\Gamma_s$, the paths $t\mapsto J_{s,t}(\psi)$ and $t\mapsto J_{s,t}(\varphi)$ do not intersect on $I_\psi^s\cup I_\varphi^s$. \par

\begin{cor}
    \label{injhomo}
    For any $s\in[0,1]$ and finitely generated subgroup $H\leq\Gamma_s$, the restriction $J_{s,t}|_{H}$ is an injective homomorphism for $t\in[0,1]$ sufficiently close to $s$.
\end{cor}

\begin{proof}
    Let $\{\varphi_1,...,\varphi_k\}$ be a finite set of generators for $H$ and set $I_*=\cap_{i=1}^k I_{\varphi_i}^s$. Lemma \ref{Lemma-j-path-props}(1) tells us each $I_{\varphi_i}^s$ is open in $[0,1]$ and contains $s$, so $I_*$ must be open in $[0,1]$ and contain $s$. Note that if $\psi_1,\psi_2\in \Gamma_s$, then $I_{\psi_1}^s\cap I_{\psi_2}^s\subseteq I_{\psi_1\psi_2}^s$, so $I_*\subseteq I_\psi^s$ for all $\psi\in H$. Thus, the claim follows from Proposition \ref{Prop-InjConds}($3\Rightarrow 2$).
\end{proof}

We will now briefly contrast this result, Corollary \ref{injhomo}, on paths in $\mathcal{D}$ with the analogous situation for convergent sequences in $\mathcal{D}$, focusing on injectivity. A classical example from J{\o}rgensen and Marden \cite{JORMAR} of conformal Dehn surgery constructs a sequence of infinite cyclic groups $\langle\psi_i\rangle$ generated by a hyperbolic converging geometrically in $\mathcal{D}$ to a rank-$2$ abelian parabolic group $\langle a,b\rangle$. The second condition for geometric convergence as given in Fact \ref{ChabProps} tells us that there exist sequences $n_i,m_i\in\Z$ such that $\psi_i^{n_i}\rightarrow a$ and $\psi_i^{m_i}\rightarrow b$ in $\mathrm{PSL}_2\C$, as $i\rightarrow\infty$. Following the framework used above to create the injective homomorphisms $J_{s,t}$ in the setting of paths, one might hope to define an injective homomorphism $J_i:\langle a,b\rangle\rightarrow\mathrm{PSL}_2\C$ for $i$ sufficiently large by setting $J_i(a)=\psi_i^{n_i}$ and $J_i(b)=\psi_i^{m_i}$. Such a map $J_i$ is not, of course, an injective homomorphism.

%%%%%%%%%%%%%%%%%%%%%%%%%%%%%%%%%%%%%%%%%%%%%%%%%%%%%%%%%%%%%%%%%%%%%%%%%%%%%%%%%%%%%%%%%%%%%%%%%%%%%%%%%%%%%%%%%%%%%%%%%%%%%%%%%%%%%%%%%%%%%%%%%%%%

\subsection{Path Components}
\label{subsec-pathcomps}

In this subsection we will show that $\mathcal{H}_\infty$ is not path connected. This contrasts with Theorem \ref{Thm-connectedcomponents}, in which we determined that $\mathcal{H}_\infty$ is connected. We will describe an infinite family of hyperbolic $3$-manifolds, which we will call symmetric infinite type $(G,N)$-glued hyperbolic $3$-manifolds, such that for each symmetric infinite type $(G,N)$-glued hyperbolic $3$-manifold $M$, the leaf \[\ell(M)=\{(M,f)\in\mathcal{H}\;|\;f\in\mathcal{F}M\}\] is a path component of $\mathcal{H}_\infty$.\par

\textbf{Construction of $(G,N)$-glued hyperbolic $3$-manifolds:} Let $N$ be a compact, oriented, connected, irreducible, atoroidal, acylindrical $3$-manifold where each of the $2n>0$ components of $\partial N$ is incompressible and has negative Euler characteristic. Fix an orientation reversing homeomorphism $\tau:\partial N\rightarrow\partial N$ such that $\tau^2=\mathrm{id}$ and $\tau$ preserves no component of $\partial N$. Observe that $\tau$ partitions the boundary components of $N$ into $n$ pairs; accordingly, assign each component of $\partial N$ a unique label, either $S_j^-$ or $S_j^+$ for $1\leq j\leq n$, so that each $\tau(S_j^-)=S_j^+$. \par

Let $G$ be a $2n$-regular directed connected graph with a labeling of the edges by $n$ colors, call the colors $\{1,...,n\}$, so that each vertex is incident to exactly one incoming edge and one outgoing edge of each color. Associate to each vertex of $G$ a copy of $N$, and suppose that $v_1,v_2\in V(G)$ are two vertices of $G$ joined by an edge labeled $j$, directed from $v_1$ to $v_2$. Then, glue the copy of $S_j^-$ in the $v_1$ copy of $N$ to the copy of $S_j^+$ in the $v_2$ copy of $N$ according to $\tau$. This gluing is depicted in Figure \ref{Figure-G-Glued} for $j=1$. Let $M$ be the manifold resulting from performing the analogous gluing at each edge of $G$. It follows from a result of Souto and Stover \cite{SOUSTO} that $M$ admits a complete hyperbolic metric. For any hyperbolic metric on $M$, we call $M$ a $(G,N)$\textbf{-glued hyperbolic }$3$\textbf{-manifold}.$\qed$

\begin{figure}
    \begin{center}
        \includegraphics[scale=0.21]{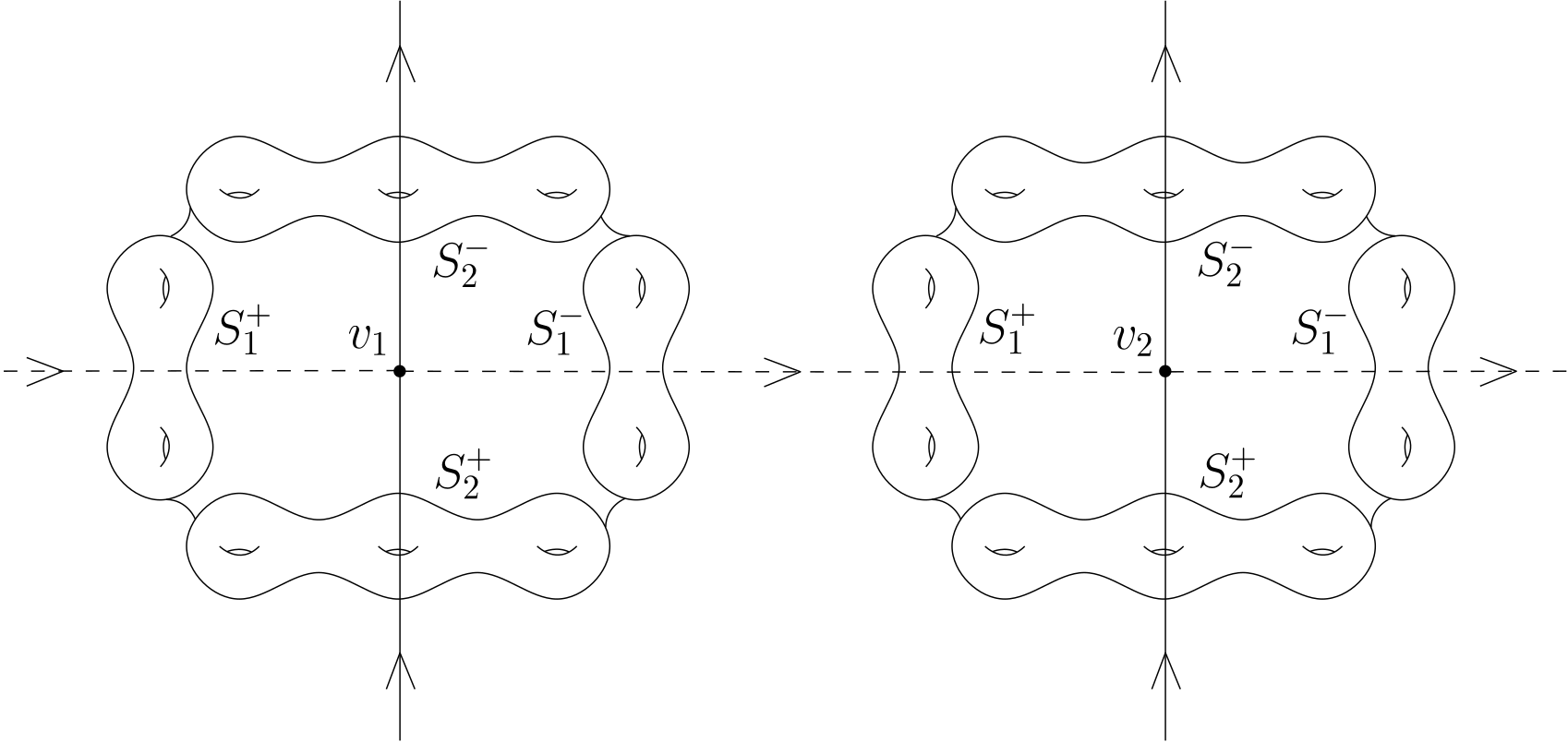}
    \end{center}
    \captionsetup{width=.8\linewidth}
    \caption{This depicts a portion of the construction of a $(G,N)$-glued hyperbolic $3$-manifold, with $G$ overlaying the copies of $N$. The edges of $G$ labeled with a $1$ are depicted as dashed and those labeled with a $2$ are depicted solid.}
    \label{Figure-G-Glued}
\end{figure}

\vspace{0.05in}

For $M$ a $(G,N)$-glued hyperbolic $3$-manifold as above, we let $\mathrm{Aut}(G)$ denote the set of graph automorphisms of $G$ that preserve edge colorings and orientations. Each $\varphi\in\mathrm{Aut}(G)$ yields an orientation preserving homeomorphism of $M$ by mapping the copy of $N$ corresponding to each vertex $v\in V(G)$ to the copy of $N$ corresponding to $\varphi(v)$; this map is a well defined homeomorphism, since $\varphi$ preserves the gluing data, so we obtain an injective homomorphism $\mathrm{Aut}(G)\hookrightarrow\mathrm{Homeo}^+(M)$ into the group of orientation preserving self homeomorphisms. We say that $M$ is \textbf{symmetric} if $\mathrm{Aut}(G)$ acts transitively on the vertex set $V(G)$, and that $M$ is \textbf{infinite type} if $|V(G)|=\infty$.\par

Note that the topological conditions on $N$ are quite general. Work of Thurston (see \cite{THUI}, \cite{KAP}) and McMullen \cite{MCM} implies that these topological conditions are jointly equivalent to $N$ admitting a hyperbolic metric with totally geodesic boundary, an even number of boundary components, and no rank-$2$ cusps. Additionally, note that Thurston's work also implies that the compact manifold $N/\tau$ obtained by gluing the boundary components of $N$ as determined by the involution $\tau$ is hyperbolizable, and Mostow rigidity implies that this hyperbolic structure is unique. \par 

There are also many examples of infinite oriented colored graphs $G$ where $\mathrm{Aut}(G)$ acts transitively on $V(G)$, for example an appropriately colored and oriented $2n$-regular tree or $n$-dimensional lattice graph.\par

In the case where $G$ is a finite graph, $M$ is compact. For such $M$, Mostow rigidity implies that $M$ has a unique hyperbolic structure and we recall from Corollary \ref{leaf-pathcomp} and Theorem \ref{Thm-connectedcomponents} that $\ell(M)$ is a path component of $\mathcal{H}$. The main results of the section are that these statements also hold in the case that $M$ is symmetric and infinite type, in which case $M$ must have infinite volume and be non-tame.\par 

\begin{theorem}
    \label{Thm-GGluePathComp}
    For any symmetric infinite type $(G,N)$-glued hyperbolic $3$-manifold $M$, the leaf $\ell(M)$ is a path component of $\mathcal{H}_\infty$.
\end{theorem}

We can determine the topology of the path component $\ell(M)$ under the hypotheses of Theorem \ref{Thm-GGluePathComp}. Recall that map $L_M:\mathcal{F}M\rightarrow\mathcal{H}$ defined by $L_M(f)=(M,f)$ is continuous by Proposition \ref{pathprop}. As discussed in Section \ref{subsec-comps}, $L_M$ descends to a continuous bijection $\mathcal{F}M/\mathrm{Isom}^+(M)\rightarrow\ell(M)$. Since $M$ is nonelementary, $\mathrm{Isom}^+(M)$ acts freely and properly discontinuously on $\mathcal{F}M$, so $\mathcal{F}M/\mathrm{Isom}^+(M)$ has a manifold topology inherited from $M$. In the case where $M$ is a symmetric infinite type $(G,N)$-glued hyperbolic, as in Theorem \ref{Thm-GGluePathComp}, Proposition \ref{Prop-GGluedUnique} (below) implies that $\mathcal{F}M/\mathrm{Isom}^+(M)$ is compact. It then follows that the path component $\ell(M)$ with the subspace topology is homeomorphic to the manifold $\mathcal{F}M/\mathrm{Isom}^+(M)$. \par

Theorem \ref{Thm-GGluePathComp} gives the following immediate corollary.

\begin{maincor}
\label{cor-notpathconn}
    $\mathcal{H}_\infty$ is not path connected.
\end{maincor}

Before continuing, it is worthwhile contrasting Theorem \ref{Thm-GGluePathComp} with Example \hyperlink{Ex-NonTame}{B.3}, which will show that the conclusion does not hold if some of the hypotheses are weakened. Let $S$ be a closed surface of genus at least $2$ and let $C_0\subseteq S$ be a pants decomposition of $S$. Let $\tau:S\rightarrow S$ be a pseudo-Anosov homeomorphism and set \[C=\bigcup_{n\in \Z}\tau^n(C_0)\times\{n+(1/2)\}\subseteq S\times\R.\] In Example \hyperlink{Ex-NonTame}{B.3}, a path in $\mathcal{H}$ is constructed between a geometrically finite framed hyperbolic $3$-manifold and a framed manifold $M_{\infty}$, which is homeomorphic to $(S\times\R)\backslash C$, and has a unique hyperbolic structure.\par

Let $T_0\subseteq S\times [0,1]$ be a small open tubular neighborhood of $C_0\times\{1/2\}$ so that the interior of the $N'=(S\times[0,1])\backslash T_0$ is homeomorphic to $(S\times(0,1))\backslash(C_0\times\{1/2\})$. Let $S_-'=S\times\{0\}\subseteq N'$ and $S_+'=S\times\{1\}\subseteq N'$. For $n\in\Z$, enumerate copies $N_n'$ of $N'$, and let $M'$ be the result of using the homeomorphism $\tau$ to glue the copy of $S_-'$ in $N_n'$ to the copy of $S_+'$ in $N_{n+1}'$ for all $n\in\Z$. Observe that the interior of $M'$ is homeomorphic to the framed hyperbolic $3$-manifold $M_\infty$, which Example \hyperlink{Ex-NonTame}{B.3} and Lemma \ref{pathcomp} tell us is in the same path component of $\mathcal{H}_\infty$ as $(\mathbb{H}^3,\mathcal{O})$.\par

The construction of $M'$ is very similar to the construction of a $(G,N)$-glued hyperbolic $3$-manifold, where $G$ is the bi-infinite path graph. The primary differences are that $N'$ is not acylindrical and has torus boundary components. Indeed, $N'$ contains multiple essential annuli, each with one boundary components on a boundary torus and the other boundary component on $S'_{\pm}$, homotopic to the corresponding curve of $C_0$. The key step in the construction of the paths in Examples \hyperlink{Ex-CuspRank}{B.2} and \hyperlink{Ex-NonTame}{B.3} where new parabolic isometries are brought in from infinity is exactly where these annuli and tori are are arising.\par

We will need some preliminary results before we can prove Theorem \ref{Thm-GGluePathComp}. The first is a strong rigidity result for symmetric $(G,N)$-glued hyperbolic $3$-manifolds, which also tells us that the topological symmetry of a symmetric $(G,N)$-glued hyperbolic $3$-manifold descends to symmetry of the hyperbolic metric. This result may be well known to experts, and is similar to the main theorem of \cite{BrockMinskyNamaziSouto}, which uses techniques different from those here to address gluings of hyperbolizable manifolds containing essential cylinders and compressing disks. This is a version of a result to appear in a forthcoming paper of Cremaschi and Yarmola \cite{CREMYAR}, and the author thanks them for allowing it to appear here.

\begin{prop}
\label{Prop-GGluedUnique}
    For any Kleinian group $\Gamma$ such that $M=\mathbb{H}^3/\Gamma$ is a symmetric $(G,N)$-glued hyperbolic $3$-manifold, $\AH(\Gamma)$ consists of two points. Further, there is an injective homomorphism $\mathrm{Aut(G)}\hookrightarrow\mathrm{Isom}^+(M)$ so that $M/\mathrm{Aut(G)}=N/\tau$.
\end{prop}

\begin{proof}
    Fix $p\in M$ and consider $\rho\in D(\Gamma)$. Since any hyperbolic manifold is a $K(\pi,1)$ space, there exists a homotopy equivalence $h:\mathbb{H}^3/\Gamma\rightarrow\mathbb{H}^3/\rho(\Gamma)$ such that $\Psi_\rho\circ h_*=\rho\circ \Psi$, where $\Psi$ (resp. $\Psi_\rho$) is the holonomy representation of $\pi_1(M)$ (resp. $\pi_1(\mathbb{H}^3/\rho(\Gamma))$) determined by $\Gamma$ and a $\Gamma$-lift of $p$ (resp. $\rho(\Gamma)$ and an appropriate $\rho_n(\Gamma)$-lift of $h(p)$). A result of Souto and Stover \cite[Lemma 3.3]{SOUSTO} implies that $h$ must be homotopic to a homeomorphism, so we may assume that $h$ itself is a homeomorphism. We may also assume without loss of generality that $h$ is orientation preserving: otherwise, replace $\rho$ with $\rho'\in D(\Gamma)$ where $\rho'$ is a representation so that $\mathbb{H}^3/\rho'(\Gamma)$ is isometric to $\mathbb{H}^3/\rho(\Gamma)$ by an orientation reversing isometry.\par

    Let $N_0$ denote one of the copies of $N$ in the construction of $M$ as a $(G,N)$-glued hyperbolic $3$-manifold, and assume $p\in N_0$. Let $H=\Psi(\pi_1(N_0,p))\leq\Gamma$ denote the subgroup corresponding to $\pi_1(N_0,p)$ under the holonomy representation. It follows from Canary's covering theorem \cite[Corollary C]{DICKCOVER} that $H$ and $\rho(H)$ must be convex cocompact. Since $h$ is orientation preserving and $h(N_0)$ is homeomorphic to $N_0$ by the preceding paragraph, we see $\rho|_H\in D_{\GF}(H)$. Thus, we have a map $\AH(\Gamma)\rightarrow \GF(H)$ defined by $[\rho]\mapsto[\rho|_H]$. By work of Ahlfors, Bers, Marden, and Sullivan (see eg. \cite{MAT}) we may identify $\GF(H)$ with the Teichm\"{u}ller space $\mathcal{T}(\partial N_0)$. By composition, we obtain a map $\chi:\AH(\Gamma)\rightarrow\mathcal{T}(\partial N_0)$. Our goal will be to show that $\chi$ is a constant map, at which point the two claims will follow from the topological symmetry of $M$.\par

     As in the above construction of $(G,N)$-glued hyperbolic $3$-manifolds, enumerate the components of $\partial N_0$ by $S_j^-$ and $S_j^+$ for $1\leq j\leq n$, so that the involution $\tau$ in the construction satisfies $\tau(S_j^-)=S_j^+$ for each $j$. Let $N_j^-$ and $N_j^+$ denote the copies of $N$ to which $N_0$ is glued along $S_j^-$ and $S_j^+$, respectively. Note that we may identify \[\mathcal{T}(\partial N_0)=\prod_{1\leq j\leq n}\mathcal{T}(S_j^-)\times\mathcal{T}(S_j^+),\;\;\;\;\;\mathcal{T}(\overline{\partial N_0})=\prod_{1\leq j\leq n}\mathcal{T}(\overline{S_j^-})\times\mathcal{T}(\overline{S_j^+})\] where $\overline{S}$ denotes the surface $S$ with the reversed orientation. From the gluing map $\tau$, we obtain Teichm\"{u}ller isometries $\tau_j^-:\mathcal{T}(S_j^-)\rightarrow\mathcal{T}(\overline{S_j^+})$ and $\tau_j^+:\mathcal{T}(S_j^+)\rightarrow\mathcal{T}(\overline{S_j^-})$. Let $Y=\chi(\AH(\Gamma))$, and let $D_j^-$ and $D_j^+$ denote the projections of $Y$ to the $\mathcal{T}(S_j^-)$ and $\mathcal{T}(S_j^+)$ factors of $\mathcal{T}(\partial N_0)$, respectively.\par

     Fix $\mu\in D_j^-$, for some $1\leq j\leq n$. Let $H'\leq\Gamma$ be a subgroup representing the conjugacy class corresponding to $\pi_1(N_j^+)$ under the holonomy representation. The symmetry of $M$ implies that for some $\rho\in D(\Gamma)$, $\mu$ appears as the conformal boundary of the end of $\mathbb{H}^3/\rho(H')$ corresponding to the copy of $S_j^-$ appearing in $N_j^+$. Let $\sigma_{N_0}:\mathcal{T}(\partial N_0)\rightarrow\mathcal{T}(\overline{\partial N_0})$ denote Thurston's skinning map corresponding to $N_0$ (see eg. \cite{KENT}) and observe that $\tau_j^-(\mu)$ is the $\mathcal{T}(\overline{S_j^+})$ coordinate of $\sigma_{N_0}(\chi(\rho))$, since $N_0$ is attached to $N_j^+$ by gluing the $N_0$ copy of $S_j^+$ to the $N_j^+$ copy of $S_j^-$. Therefore, $\tau_j^-(\mu)$, and hence $\tau_j^-(D_j^-)$, is contained in the projection of $\sigma_{N_0}(Y)$ to $\mathcal{T}(\overline{S_j^+})$. A similar argument establishes that each $\tau_j^i(D_j^+)$ is contained in the projection of $\sigma_{N_0}(Y)$ to $\mathcal{T}(\overline{S_j^-})$.\par

     Since $N_0$ is acylindrical, the bounded image theorem \cite{KENT} implies that $\sigma_{N_0}(Y)$ is bounded. Hence, it follows from the preceding paragraph that each $d_j^-=\mathrm{diam}_{\mathcal{T}}(D_j^-)$ and $d_j^+=\mathrm{diam}_{\mathcal{T}}(D_j^+)$ are finite, where $\mathrm{diam}_\mathcal{T}$ is the diameter of these subsets in their respective Teichm\"{u}ller spaces, with respect to the Teichm\"{u}ller metric. If any $d_k^->0$ or $d_k^+>0$, then McMullen's strict contracting of the skinning map \cite{MCM} implies that for each $j$, \[d_j^- < \max\{d_1^-,...,d_n^-,d_1^+,...,d_n^+\}\;\;\;\;\;\text{ and }\;\;\;\;\; d_j^+ < \max\{d_1^-,...,d_n^-,d_1^+,...,d_n^+\}\]  again by the preceding paragraph. Not all of these inequalities can hold, so we conclude that $d_j^-=d_j^+=0$ for all $j$, which completes the proof that $\chi$ is constant.
\end{proof}

We will also need a result on discrete extensions of Kleinian groups. Several authors, for example see Susskind \cite{SUSS} and Anderson \cite{ANDER}, have examined sets of the form \[\{\psi\in\mathrm{PSL}_2\C \;|\; \langle\Gamma,\psi\rangle\text{ is discrete \& }\mathrm{fix}(\psi)\cap\Lambda(\Gamma)\neq\varnothing\},\] where $\mathrm{fix}(\psi)\subseteq S^2_\infty$ is the set of points fixed by $\psi$, placing various constraints on the Kleinian group $\Gamma$, or additional restrictions on $\psi$. Most of these results qualitatively describe the elements $\psi$ in such a set. Our primary focus will be the cardinality of such a set, though some details on the elements will be gleaned along the way.\par

For a Kleinian group $\Gamma$, define \[C_\Gamma=\{\psi\in\mathrm{PSL}_2\C\;|\;\psi\Gamma\psi^{-1}\cap\Gamma\text{ is nonelementary}\}.\]

\begin{lemma}
    \label{Lemma-CCountable}
    For any Kleinian group $\Gamma$, $C_\Gamma$ is countable.
\end{lemma}

\begin{proof}
    Define a map $F:\Gamma^4\rightarrow C_\Gamma$ as follows. Fix $\varphi_1,\varphi_2,\varphi_3,\varphi_4\in\Gamma$. If $\langle\varphi_1,\varphi_2\rangle$ is nonelementary and there exists $\psi\in\mathrm{PSL}_2\C$ such that $\varphi_3=\psi\varphi_1\psi^{-1}$ and $\varphi_4=\psi\varphi_2\psi^{-1}$, then Fact \ref{Fact-UniqueConjugator} tells us $\psi$ is uniquely determined, and we define $F(\varphi_1,\varphi_2,\varphi_3,\varphi_4)=\psi$. Otherwise, define $F(\varphi_1,\varphi_2,\varphi_3,\varphi_4)=\mathds{1}$. Then, $F$ has countable domain and is surjective, confirming that $C_\Gamma$ is countable.
\end{proof}

The following proposition will provide the result we need on discrete extensions of Kleinian groups.

\begin{prop}
\label{Thm-DiscCountable}
    Suppose $\Gamma$ is a purely hyperbolic Kleinian group such that the injectivity radius of $\mathbb{H}^3/\Gamma$ is uniformly bounded above, and let \[S=\{\psi\in\mathrm{PSL}_2\C \;|\; \langle\Gamma,\psi\rangle\text{ is discrete}\}.\] Then, for each $\psi\in S$, there is some $n\in\Z$ such that $\psi^n\in C_\Gamma$. In particular, $S$ is countable. 
\end{prop}

Note that the hypothesis that $\mathbb{H}^3/\Gamma$ has injectivity radius bounded above implies that $\Lambda(\Gamma)= S^2_\infty$, so the fixed point set of every torsion free $\psi\in\mathrm{PSL}_2\C$ intersects $\Lambda(\Gamma)$. This result may be of independent interest due to the lack of any restrictions on $\psi$ in the definition of the set $S$.

\begin{proof}[Proof of Proposition \ref{Thm-DiscCountable}]
    Fix $\psi\in S$. First, assume that $\mathrm{fix}(\psi)\cap\mathrm{fix}(\varphi)\neq\varnothing$ for some $\varphi\in\Gamma$. Then, the discreteness of $\langle\Gamma,\psi\rangle$ implies that $\langle\psi,\varphi\rangle$ is virtually cyclic, so there is some $n\in\Z$ such that $\psi^n\in\langle\varphi\rangle\leq\Gamma\subseteq C_\Gamma$.\par

    Now, assume that $\mathrm{fix}(\psi)\cap\mathrm{fix}(\varphi)=\varnothing$ for all $\varphi\in\Gamma$. Note that this implies that for all $\varphi\in\Gamma$ and $n\neq 0$, $\langle\varphi,\psi^n\varphi\psi^{-n}\rangle$ is a nonelementary discrete group.\par
    
    Let $R$ be an upper bound for the injectivity radius of $\mathbb{H}^3/\Gamma$ and fix $p\in\mathbb{H}^3$. Then, for all $n\in\N$, there exists $\varphi_n\in\Gamma\backslash\{\mathds{1}\}$ such that $d_{\mathbb{H}^3}(\varphi_n\cdot(\psi^n\cdot p),\psi^n\cdot p)\leq R$, and hence that \[d_{\mathbb{H}^3}((\psi^{-n}\varphi_n\psi^n)\cdot p,p)\leq R.\] The set $\{\psi^{-n}\varphi_n\psi^n \;|\; n\in\N\}$ must be bounded, so the discreteness of $\langle\Gamma,\psi\rangle$ implies that this set is finite. Enumerate \[\{\psi^{-n}\varphi_n\psi^n \;|\; n\in\N\}=\{a_1,...,a_N\}\] for some $N\geq 1$ and let $S_i=\{n\in\N \;|\; \psi^{-n}\varphi_n\psi^n=a_i\}$, for $1\leq i\leq N$.\par

    For each $i$, let $m_i=\min S_i$. Then, for $n\in S_i$, we have 
    \begin{equation*} 
        \begin{split}
            \psi^{-n}\varphi_n\psi^n&=a_i\\
            \varphi_n&=\psi^na_i\psi^{-n}\\
            \varphi_n&=\psi^{n-m_n}(\psi^{m_i}a_i\psi^{-m_i})\psi^{m_i-n}.
        \end{split}
    \end{equation*}
    Thus, setting $S_i'=\{n-m_i:n\in S_i\}$ and renaming $\varphi_i^*=\psi^{m_i}a_i\psi^{-m_i}=\varphi_{m_i}\in\Gamma$, we have $\psi^n\varphi_i^*\psi^{-n}\in\Gamma$ for all $n\in S_i'$.\par

    Since the sets $S_1,...,S_N$ partition $\N$, van der Waerden's theorem \cite{VAN} tells us that one of the $S_i$, say $S_1$, contains an arithmetic progression of length $3$. Then, $S_1'$ contains an arithmetic progression of length $3$, so there exist $A,D\in\N$ such that $A,A+D,A+2D\in S_1'$. We may therefore define the subgroups
    \begin{equation*} 
        \begin{split}
            H_1 &= \langle\psi^A\varphi_1^*\psi^{-A},\psi^{A+D}\varphi_1^*\psi^{-(A+D)}\rangle\leq\Gamma \;\;\;\text{ and}\\
            H_2 &= \langle\psi^{A+D}\varphi_1^*\psi^{-(A+D)},\psi^{A+2D}\varphi_1^*\psi^{-(A+2D)}\rangle\leq\Gamma.
        \end{split}
    \end{equation*}
    As noted, the assumption on $\psi$ implies that each $H_1,H_2$ is nonelementary. We can see that $\psi^DH_1\psi^{-D}=H_2$, so $\psi^D\in C_\Gamma$, as desired.\par

    Note that for any countable $X\subseteq\mathrm{PSL}_2\C$, the set $\{\varphi\in\mathrm{PSL}_2\C \;|\; \exists n\in\Z \text{ such that }\varphi^n\in X\}$ is countable, so the claim that $S$ is countable now follows from Lemma \ref{Lemma-CCountable}.
\end{proof}

Finally, we can complete the proof of Theorem \ref{Thm-GGluePathComp}.

\begin{proof}[Proof of Theorem \ref{Thm-GGluePathComp}]
    Let $\Gamma:[0,1]\rightarrow\mathcal{D}$ be a path and write $\Gamma_t=\Gamma(t)$ for all $t$. Suppose that $M=\mathbb{H}^3/\Gamma_0$ is a symmetric infinite type $(G,N)$-glued hyperbolic $3$-manifold. We will show that $\Gamma_t$ is conjugate to $\Gamma_0$ for all $t$, which will imply the result.\par

    Fix some vertex $v_0$ in $G$. Note from the construction of $(G,N)$-glued hyperbolic $3$-manifolds that we can write $M=\bigcup_{n\geq 0} X_n$, where each $X_n$ is the union of the copies of $N$ that correspond to vertices of $G$ with graph distance at most $n$ from $v_0$ (the graph distance being measured disregarding edge orientation). Fix a point $p\in X_0$, and for $n\geq 0$ let $H_n\leq \Gamma_0$ be the subgroup that is the image of $\pi_1(X_n,p)\leq \pi_1(M,p)$ under a fixed holonomy representation $\pi_1(M,p)\rightarrow \Gamma_0$ with respect to $p$ and $\Gamma_0$.\par

     For each $n\geq0$, let \[t_n=\sup\{t\in[0,1]\;|\; \text{for all }t'\leq t\text{, }J_{0,t'}|_{H_n}\text{ is finite valued}\}\] where the maps $J_{s,t}:\Gamma_s\rightarrow\overline{\mathrm{PSL}_2\C}$ are those constructed in Equation \ref{eqn-JDef} with respect to the path $\Gamma$. Since each $H_n$ is finitely generated, Corollary \ref{injhomo} tells us each $t_n>0$. \textit{We assume for now that $J_{0,t}|_{H_0}$ is in fact finite valued for all $t\in[0,1]$, so $t_0=1$.}\par

     Suppose, for contradiction, that there exists some $n\geq 1$ such that $t_n<1$. Observe that $J_{0,t_n}|_{H_n}$ must not be finite valued, else Corollary \ref{injhomo} combined with Proposition \ref{Prop-InjConds} would provide a contradiction to the definition of $t_n$. For each $t\in[0,t_n)$, Proposition \ref{Prop-InjConds} implies that $J_{0,t}|_{H_n}$ is an injective homomorphism, so  $J_{0,t}|_{H_n}\in D(H_n)$. Let $\{s_k\}_{k=1}^\infty\subseteq[0,t_n)$ be a sequence such that $s_k\rightarrow t_n$ as $k\rightarrow\infty$. Since each $X_n$ is acylindrical, Thurston's ``$\AH(\text{acylindrical})$ is compact" theorem \cite{THUI} tells us that $\{[J_{0,s_k}|_{H_n}]\}\subseteq \AH(H_n)$ has a convergent subsequence, which we will denote the same way. Note that $J_{0,s_k}|_{H_0}\rightarrow J_{0,t_n}|_{H_0}$ in $D(H_0)$, as $k\rightarrow\infty$, so since $H_0$ is nonelementary, Lemma \ref{Lemma-AHNonelemSubgroupConvg} tells us there exists $\rho\in D(H_n)$ such that $J_{0,s_k}|_{H_n}\rightarrow\rho$ as $k\rightarrow\infty$. In particular, for all $\varphi\in H_n$, $J_{0,s_k}(\varphi)$ is bounded as $k\rightarrow\infty$, so it follows that the path $t\mapsto J_{0,t}(\varphi)$ does not diverge to $\infty$ as $t$ increases from $0$ to $t_n$. This, however, contradicts that $J_{0,t_n}|_{H_n}$ is not finite valued, by the continuity statement in Lemma \ref{Lemma-JFacts}. We therefore conclude that $t_n=1$ for all $n$ and that $J_{0,t}$ is injective for all $t\in[0,1]$.\par

     Since each $J_{0,t}$ is injective, $t\mapsto[J_{0,t}]$ is a path in $\AH(\Gamma_0)$. Since $M$ is a symmetric $(G,N)$-glued hyperbolic $3$-manifold, the rigidity given by Proposition \ref{Prop-GGluedUnique} combined with the continuity statement in Fact \ref{Fact-UniqueConjugator} tell us that there exists a path $\psi:[0,1]\rightarrow\mathrm{PSL}_2\C$ so that for all $\varphi\in\Gamma_0$, $\psi(t) J_{0,t}(\varphi)\psi(t)^{-1}=\varphi$. Now, define a path $\Gamma':[0,1]\rightarrow\mathcal{D}$ by $\Gamma'(t)=\psi(t)\Gamma_t\psi(t)^{-1}$, and write $\Gamma_t'=\Gamma'(t)$. For $s,t\in[0,1]$, let $J_{s,t}':\Gamma_s'\rightarrow\overline{\mathrm{PSL}_2\C}$ denote the maps defined in Equation \ref{eqn-JDef}, now corresponding to the path $\Gamma'$; similarly, let $(I_\varphi^{s})'$ denote the maximal intervals constructed there, corresponding to the path $\Gamma'$.\par

     Notice that $J_{0,t}'$ is the identity on each element of $\Gamma_0'=\Gamma_0$, and in particular $J_{0,t}'(\Gamma_0')=\Gamma_0$ for all $t$. Suppose that there exists $\varphi\in\Gamma_s'\backslash J_{0,s}'(\Gamma_0')$ for some $s\in[0,1]$. Observe that $0\not\in (I_\varphi^{s})'$, else $J_{0,s}'(J_{s,0}'(\varphi))=J_{s,s}(\varphi)=\varphi$ by Lemma \ref{Lemma-JFacts}(2), contradicting that $\varphi$ is not in the image of $J_{0,s}'$. Thus, Lemma \ref{Lemma-JFacts}(1) tells us that the path $J_{s,t}(\varphi)$ diverges to infinity in $\mathrm{PSL}_2\C$ as $t$ decreases from $s$ towards $0$. In particular, we see that $J_{s,t}(\varphi)$ must take uncountably many values as $t$ decreases. Note that the second statement of Proposition \ref{Prop-GGluedUnique} implies that the compact copies of $N$ in $\mathbb{H}^3/\Gamma_0$ are all isometric, so $\mathbb{H}^3/\Gamma_0$ has bounded injectivity radius; $\Gamma_0$ is a purely hyperbolic Kleinian group, so we have obtained a contradiction to Proposition \ref{Thm-DiscCountable}. Thus, $J_{0,t}'(\Gamma_0')=\Gamma_0'$ for all $t$, so we have shown that there exists a path $\psi:[0,1]\rightarrow\mathrm{PSL}_2\C$ such that $\Gamma_0=\psi(t)\Gamma_t\psi(t)^{-1}$ for all $t\in[0,1]$.\par

     Finally, we need to address our added assumption (in italics) that $J_{0,t}|_{H_0}$ is finite valued for all $t\in[0,1]$. For a general path $\Gamma:[0,1]\rightarrow\mathcal{D}$ where $M=\mathbb{H}^3/\Gamma_0$ is a symmetric $(G,N)$-glued hyperbolic $3$-manifold, we know that $t_0>0$, so we have shown that $\Gamma_t$ must be conjugate to $\Gamma_0$ for $t$ sufficiently close to $0$. Let $t_*\in[0,1]$ be the supremal value such that for all $t<t_*$, $\Gamma_t$ is conjugate to $\Gamma_0$. Suppose for contradiction that $t_*<1$. The symmetry of $M$ described in Proposition \ref{Prop-GGluedUnique} implies that $\ell(M)\subseteq\mathcal{H}$ is compact, or equivalently that the subspace of $\mathcal{D}$ consisting of Kleinian groups that are conjugate to $\Gamma_0$ is compact. Thus, it must be the case that $\Gamma_{t_*}$ is conjugate to $\Gamma_0$. Hence, $\mathbb{H}^3/\Gamma_{t_*}$ is a symmetric infinite type $(G,N)$-glued hyperbolic $3$-manifold, so applying the above result to the path with the restricted domain $\Gamma|_{[t_*,1]}$, we see that $\Gamma_t$ must be conjugate to $\Gamma_{t_*}$, and hence to $\Gamma_0$, for $t$ sufficiently close to $t_*$. This contradicts the definition of $t_*$, so we conclude that $t_*=1$, as was to be shown.
\end{proof}

\printbibliography

\end{document}